\documentclass{amsart}
\usepackage{amsmath,amssymb,amsthm,amsfonts,amscd,mathrsfs,tikz-cd}
\usepackage{listings}
\usepackage[all,cmtip]{xy}
\usepackage{graphicx}
    \usepackage{tikz}
    \usetikzlibrary{decorations.markings,arrows,automata}

\theoremstyle{plain}

\newtheorem{theorem}{Theorem}[section]
\newtheorem{proposition}[theorem]{Proposition}
\newtheorem{lemma}[theorem]{Lemma}
\newtheorem{corollary}[theorem]{Corollary}

\newtheorem*{proposition*}{Proposition}

\theoremstyle{definition}
\newtheorem{definition}[theorem]{Definition}
\newtheorem{example}[theorem]{Example}

\theoremstyle{remark}
\newtheorem{remark}[theorem]{Remark}

\newcommand{\defref}[1]{Definition~\ref{#1}}

\newcommand{\be}{\begin{enumerate}}
\newcommand{\ee}{\end{enumerate}}

\newcommand{\Z}{\mathbb{Z}}

\newcommand{\lk}{\mathrm{lk}}

\begin{document}

\title{The Digital Hopf Construction}

\author{Gregory Lupton}
\author{John Oprea}
\author{Nicholas A. Scoville}

\address{Department of Mathematics, Cleveland State University, Cleveland OH 44115 U.S.A.}

\email{g.lupton@csuohio.edu}
\email{j.oprea@csuohio.edu}

\address{Department of Mathematics and Computer Science, Ursinus College, Collegeville PA 19426 U.S.A.}

\email{nscoville@ursinus.edu}

\date{\today}


\begin{abstract}Various concepts and constructions in homotopy theory have been defined in the digital setting. Although there have been several attempts at a definition of a fibration in the digital setting, robust examples of these digital fibrations are few and far between.  In this paper, we develop a digital Hopf fibration within the category of tolerance spaces. By widening our category to that of tolerance spaces, we are able to give a construction of this digital Hopf fibration which mimics the smooth setting.
\end{abstract}

\thanks{This work was partially supported by grants from the Simons Foundation: (\#209575 to Gregory Lupton
and \#244393 to John Oprea).}

\maketitle

\section{Introduction}

In recent years attempts have been made by various authors to introduce concepts from algebraic topology into the study of digital images. In particular, there have been attempts to study both digital fibrations \cite{EgeKaraca2017} and digital covering spaces \cite{BoxerKaraca2008} \cite{BoxerKaraca2010}. The notion of digital fibration, however, is problematic since digital topology is very rigid with respect to the usual notion of homotopy and certainly with respect to the homotopy lifting property. A strict translation of the definition of a classical fibration to the digital setting seems to yield  relatively uninteresting examples of fibrations. A recent paper by the authors \cite{LOS19a} has relaxed the classical definition in order to obtain more meaningful examples.  In this paper, we introduce another new approach to the idea of digital fibrations in two new ways. First, we use a notion of \emph{simple digital homotopy} as defined in \cite{Evako2015} and second, we expand our point of view to that of \emph{tolerance spaces}. With this in mind, we confine ourselves to proposing a digital analogue of the \emph{Hopf fibration}, perhaps the most important single example in the history of algebraic topology. Furthermore, this proposed construction is precisely the digital analogue of fundamental topological constructions, which in turn leads to meaningful general definitions of concepts such as suspension, cone, and fibre bundle. The classical Hopf fibration is defined by using the multiplication $\mu\colon S^1\times S^1\to S^1$ which gives a map $S^1*S^1\to S^2=\Sigma S^1$ defined by $[x,y,t]\mapsto [\mu(x,y),t]$.  We will see in Proposition \ref{prop: no circle multiplication} that no such continuous multiplication on a circle is possible in the digital setting.  Our workaround for this problem is to define a multiplication-like map that takes in two points on an $8$-point circle and yields a point on a $4$-point circle.  The main result is Theorem \ref{thm: digital Hopf fiber bundle} where we show that this is a digital fiber bundle.

An important point that seems unescapable in this context is that these constructions often lead us to work with objects that are not digital images as they are usually construed. These objects might be simply called \emph{sets with an adjacency relation}. More precisely, these sets make up a category of \emph{tolerance spaces} which was defined by Zeeman \cite{Zeeman} as well as Poincar\'{e}  (called by him ``physical continua").  Tolerance spaces have been studied by authors in recent years as well \cite{Peters12}, \cite{Sossinsky1986}, \cite{MuirWarner1980}. A tolerance space is a set together with a symmetric, reflexive relation called the tolerance and the idea is that points within the tolerance are indistinguishable. Indeed, Tim Poston \cite{Poston} referred to these as ``fuzzy spaces" and the study of these fuzzy spaces as ``fuzzy geometry." Poston also argued that the tolerance space point of view is more appropriate for studying the physical world since ``distinctions depending on an infinite number of decimal places \ldots become absurd when considered as physical statements \ldots [giving] rise to a relation of 'indistinguishability' \ldots [and] unlike the Euclidean plane or pseudo-Riemannian manifolds, then, fuzzy spaces occur as objects of direct experience."  This ``homotopy'' viewpoint is also important for us because our homotopy-inspired constructions lead us to ``fattened" spheres, which may be thought of as fuzzy versions of spheres. It is also true that, whenever we pass from the continuous to the discrete, we must make choices about what structures to keep and what to forsake. Here, for example, we see from Theorem \ref{thm no tolerance group} that the only digital groups are simple digitally contractible ones, which severely limit the number of interesting homotopy constructions akin to creating principal bundles with Lie structure groups.

\section{Tolerance spaces and digital images}

Recall that a \emph{digital image} $X$ \cite{Boxer2005} is a subset $X \subseteq \Z^n$ of the integral lattice in some $n$-dimensional Euclidean space along with a notion of adjacency (nearness) between points that is determined by the coordinates.  More generally, we may consider a set along with an adjacency relation.

\begin{definition}\label{def:tolerance space}
A \emph{tolerance space} consists of a set $X$ and a symmetric, reflexive relation denoted by $\sim$. We say that $x,y \in X$ are \emph{within the tolerance} or \emph{adjacent} if $x\sim y$. For a fixed $x$, the number of adjacencies of $x$ (excluding $x$'s adjacency to itself) is known as the \emph{valence} of $x$. A  map between tolerance spaces $f\colon X \to Y$ is \emph{continuous} if $f(x)$ and $f(y)$ are within the tolerance in $Y$ whenever $x$ and $y$ are within the tolerance in $X$. Under this definition, tolerance spaces form a category \cite[Chapter 3]{Sossinsky1986}. In this paper, we only consider finite tolerance spaces.
\end{definition}

\begin{remark} A tolerance space may be viewed as a graph with each node representing an element of the set and an edge connecting two nodes if and only if the corresponding elements are adjacent (suppressing the adjacency of a point with itself). Of course, a digital image may also be viewed as a graph. From now on we will use a tolerance space and its corresponding graph representation interchangeably.
\end{remark}

Given Definition \ref{def: digital image}, it is easy to see that digital images are just a special kind of tolerance space.  More surprisingly, every tolerance space may be embedded as a digital image.  We show this below in Proposition \ref{prop: tolerance is digital image}.

\begin{definition}\label{def:products}
The product of tolerance spaces $X$ and $Y$ is the Cartesian product of sets $X \times Y$ along with the adjacency relation $(x, y)$ adjacent to
$(x', y')$ when $x$ and $x'$ are either equal or adjacent in $X$, and $y$ and $y'$ are either equal or adjacent in $Y$. Clearly $X\times Y$ is a tolerance space.
\end{definition}

\begin{definition}\label{def: digital image} A \emph{digital image} or \emph{digital space} is any subset $A\subseteq \Z^n$ where $\Z$ is the tolerance space in which $\Z^n$ is given the product structure in Definition \ref{def:products}.
\end{definition}

In the literature, it is common to allow for different adjacencies in $\Z^n$. We do not allow this, however, for technical reasons.  Instead, we fix once and for all the adjacencies in $\Z^n$ given by \defref{def: digital image}. This definition is tantamount to assuming that $\Z^n$, and any digital image in it, has the highest degree of adjacency possible ($8$-adjacency in $\Z^2$, $26$-adjacency in $\Z^3$, etc.).

Although it is clear that the product of digital images is a digital image, it is not so clear that the tolerance spaces which result from our constructions in Section \ref{sec:cones} are in fact digital images e.g. the cone or suspension of a digital image.  However, the following result guarantees that any tolerance space is in fact able to be embedded as a digital image.\footnote{The authors wish to thank Andrea Bianchi for this result.}

\begin{proposition}\label{prop: tolerance is digital image}
Let $G$ be a finite simple graph (no multiedges and no loops from a vertex to itself).  Then $G$ may be embedded as a digital image with vertices in the hypercube $[-1, 1]^{n-1} \subseteq \Z^{n-1}$, where $n = |G|$, the number of vertices.
\end{proposition}

\begin{proof}
Work by induction on $n$.  Induction starts with $n = 2$ where there is nothing to show.

Inductively assume that, if $|G| \leq n$, then we may embed $G$ as a digital image in  $[-1, 1]^{n-1}$.  Suppose we have a graph $G'$ with $n+1$ vertices.  Choose any vertex $x \in G'$ and write $G' = G \cup \{x\}$ with $|G| = n$.  Embed $G$ as a digital image in  $[-1, 1]^{n-1} \subseteq \Z^{n-1} \subseteq \Z^{n-1} \times \Z = \Z^{n}$. Then each vertex $y \in G$ has coordinates $y = (y_1, \ldots, y_{n-1}, 0) \in \Z^{n}$, and we have $y_i \in \{\pm1, 0\}$ for $i = 1, \ldots, n-1$.  Now separate $G$ into the disjoint union $G = \lk(x) \sqcup \lk(x)^C$ (set-wise, there will be adjacencies across the union, but that's OK). For each $y \in \lk(x)^C$, move it down to the plane $y_n = -1$.  In other words, adjust the embedding of $G$ in $\Z^n$ using the isomorphism of digital images $\phi\colon G \to \overline{G}$ given by
$$\phi(y_1, \ldots, y_{n-1}, 0) = \begin{cases} (y_1, \ldots, y_{n-1}, 0) & \text{if } y \in \lk(x) \\  (y_1, \ldots, y_{n-1}, -1) & \text{if }  y \in \lk(x)^C \end{cases}$$
This is an isomorphism, since we have---for $y, y' \in \Z^{n-1} \times \{0\} \subseteq \Z^n$---
$$y \sim_{\Z^n} y' \iff (y_1, \ldots, y_{n-1}) \sim_{\Z^{n-1}} (y'_1, \ldots, y'_{n-1}) \iff \phi(y) \sim_{\Z^n} \phi(y').$$
So we now have $G$ embedded in $\Z^n$ as a digital image with $\lk(x) \subseteq  [-1, 1]^{n-1} \times \{0\} \subseteq \Z^n$   and $\lk(x)^C \subseteq  [-1, 1]^{n-1} \times \{-1\} \subseteq \Z^n$.  Add $x$ as the point $x = \textbf{e}_n = (0, \ldots, 0, 1)$.  This point is adjacent to every point in  $[-1, 1]^{n-1} \times \{0\} \subseteq \Z^n$, and hence to every point of $\lk(x)$ as we have embedded it.  Furthermore,  $x = \textbf{e}_n$ is not adjacent to any point of  $[-1, 1]^{n-1} \times \{-1\} \subseteq \Z^n$, and so this produces exactly the adjacencies of $x$ from $G'$.  This completes the induction.
\end{proof}

\section{Simple digital equivalence and digital homology spheres}

\begin{definition} Let $G=(V,E)$ be a graph, $S\subseteq V$ a set of vertices.  The \emph{induced subgraph on $S$} is the graph whose set of points is $S$ and whose edges are all edges in $E$ with both endpoints in $S$.
\end{definition}

\begin{definition} Let $v\in G$ be a point.  The induced subgraph on the neighbors in $G$ of $v$ is the \emph{link of $v$}, denoted $\lk_G(v)$. When $G$ is clear from context, we sometimes write $\lk(v)$.
\end{definition}

We now recall the pertinent definitions which act as our notion of ``homotopy," due to Evako \cite{Evako94,Evako2015}, in the digital setting.

\begin{definition}\label{def: contractible} A single point is \emph{simple digitally contractible} while the empty set is not simple digitally contractible.  Inductively, suppose $G$ is a tolerance space with $n+1$ vertices and suppose that a simple digitally contractible tolerance space has been defined for any tolerance space with $1\leq i\leq n$ vertices. Suppose there is a vertex $v\in G$ such that $\lk_G(v)$ is simple digitally contractible. Then $v$ is called a \emph{simple point}. Suppose there is an edge $e=uw\in G$ such that $\lk_G(u)\cap \lk_G(w)$ is simple digitally contractible. Then $e$ is a \emph{simple edge}. Deletion of a simple vertex means removing both the simple vertex and all of its incident edges.  Attachment of a simple vertex means adding a vertex $v$ along with some set of incident edges which ensure that $v$ is simple.  Deletion of or attachment of a simple vertex or a simple edge is a \emph{simple contractible transformation}. Then $G$ is \emph{simple digitally contractible} if there is a finite sequence of simple contractible transformations from $G$ to a point. In general, $G$ and $H$ are said to be \emph{simple digitally equivalent}, denoted $G\approx H$,  if one can be obtained from the other through a series of attaching and deleting simple points and edges.
\end{definition}

%

In \cite{Evako2015}, Evako refers to this as ``homotopic" and ``digitally contractible."  However, we have chosen to add the word ``simple" in order to distinguish this from Boxer's notion \cite{Boxer2005} and other notions of digital contractibility.

\begin{example}\label{ex: complete graph} All complete graphs are simple digitally contractible. Indeed, the link of any point $v$ in a complete graph on $n$ vertices is a complete graph on $n-1$ vertices. Since the complete graph on a single point is simple digitally contractible by definition, the result follows by induction on $n$.  This fact will be used in Theorem \ref{thm no tolerance group}. By contrast, a cycle on more than $3$ points is not simple digitally contractible.
%
%
%
%
%
This is because the tolerance space has no simple points.  To see this, observe that the link of every point is two isolated points which is not simple digitally contractible since the link of both points is the empty link.
\end{example}

\begin{remark} A collection of simple points cannot be removed in any order.  That is, if $x,y\in G$ are simple, it is not necessarily the case that $y$ is simple in $G-x$.  The following example makes this clear:

$$
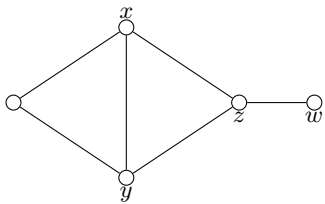
\begin{tikzpicture}[scale=.5,
    decoration={markings,mark=at position 0.6 with {\arrow{triangle 60}}},
    ]

\node[inner sep=2pt, circle] (a) at (-3,0) [draw] {};
\node[inner sep=2pt, circle] (b) at (3,0) [draw] {};
\node[inner sep=2pt, circle] (x) at (0,2) [draw] {};
\node[inner sep=2pt, circle] (y) at (0,-2) [draw] {};
\node[inner sep=2pt, circle] (w) at (5,0) [draw] {};

\draw [-] (a) -- (x)node[midway, below] {$$};
\draw [-] (a) -- (y)node[midway, below] {$$};
\draw [-] (b) -- (x)node[midway, below] {$$};
\draw [-] (b) -- (y)node[midway, below] {$$};
\draw [-] (x) -- (y)node[midway, below] {$$};
\draw [-] (w) -- (b)node[midway, below] {$$};

\node[anchor = south]  at (x) {\small{$x$}};
\node[anchor = north]  at (y) {\small{$y$}};
\node[anchor = north]  at (b) {\small{$z$}};
\node[anchor = north]  at (w) {\small{$w$}};

\end{tikzpicture}
$$
In addition, a point in the sequence to be removed may only become simple after some other points have been removed.  For example, $z$ is not simple in $G$, but $z$ is simple in $G-w$.

\end{remark}

\begin{example}\label{exam: circle hom}  As an easy example, we will show that the tolerance space $S^1_4$, a digital $1$-sphere with 4 points, is simple digitally equivalent to $S^1_8$, a digital $1$-sphere with 8 points.  This fact will be crucial in proving that we have a digital fiber bundle in Section \ref{sec: Digital Hopf bundle}. Starting with $S^1_4$, we attach two simple points, a simple edge between them, and remove a simple point to obtain $S^1_5$ below:

$
\begin{tikzpicture}[scale=1,
    decoration={markings,mark=at position 0.6 with {\arrow{triangle 60}}},
    ]

\node[inner sep=2pt, circle] (a) at (1,0) [draw] {};
\node[inner sep=2pt, circle] (b) at (0,1) [draw] {};
\node[inner sep=2pt, circle] (c) at (-1,0) [draw] {};
\node[inner sep=2pt, circle] (d) at (0,-1) [draw] {};

\draw [-] (a) -- (b)node[midway, below] {$$};
\draw [-] (b) -- (c)node[midway, below] {$$};
\draw [-] (c) -- (d)node[midway, below] {$$};
\draw [-] (d) -- (a)node[midway, below] {$$};

\node at (1.5,.5) {$\longrightarrow$};
\end{tikzpicture}
$
$
\begin{tikzpicture}[scale=1,
    decoration={markings,mark=at position 0.6 with {\arrow{triangle 60}}},
    ]

\node[inner sep=2pt, circle] (a) at (1,0) [draw] {};
\node[inner sep=2pt, circle] (b) at (0,1) [draw] {};
\node[inner sep=2pt, circle] (c) at (-1,0) [draw] {};
\node[inner sep=2pt, circle] (d) at (0,-1) [draw] {};
\node[inner sep=2pt, circle] (u) at (-1,1.5) [draw] {};
\node[inner sep=2pt, circle] (v) at (1,1.5) [draw] {};

\draw [-] (a) -- (b)node[midway, below] {$$};
\draw [-] (b) -- (c)node[midway, below] {$$};
\draw [-] (c) -- (d)node[midway, below] {$$};
\draw [-] (d) -- (a)node[midway, below] {$$};
\draw [-] (c) -- (u)node[midway, below] {$$};
\draw [-] (b) -- (u)node[midway, below] {$$};
\draw [-] (a) -- (v)node[midway, below] {$$};
\draw [-] (b) -- (v)node[midway, below] {$$};

\node at (1.5,.5) {$\longrightarrow$};
\end{tikzpicture}
$
$
\begin{tikzpicture}[scale=1,
    decoration={markings,mark=at position 0.6 with {\arrow{triangle 60}}},
    ]

\node[inner sep=2pt, circle] (a) at (1,0) [draw] {};
\node[inner sep=2pt, circle] (b) at (0,1) [draw] {};
\node[inner sep=2pt, circle] (c) at (-1,0) [draw] {};
\node[inner sep=2pt, circle] (d) at (0,-1) [draw] {};
\node[inner sep=2pt, circle] (u) at (-1,1.5) [draw] {};
\node[inner sep=2pt, circle] (v) at (1,1.5) [draw] {};

\draw [-] (a) -- (b)node[midway, below] {$$};
\draw [-] (b) -- (c)node[midway, below] {$$};
\draw [-] (c) -- (d)node[midway, below] {$$};
\draw [-] (d) -- (a)node[midway, below] {$$};
\draw [-] (c) -- (u)node[midway, below] {$$};
\draw [-] (b) -- (u)node[midway, below] {$$};
\draw [-] (a) -- (v)node[midway, below] {$$};
\draw [-] (b) -- (v)node[midway, below] {$$};
\draw [-] (u) -- (v)node[midway, below] {$$};

\node at (1.5,.5) {$\longrightarrow$};
\end{tikzpicture}
$
$
\begin{tikzpicture}[scale=1,
    decoration={markings,mark=at position 0.6 with {\arrow{triangle 60}}},
    ]

\node[inner sep=2pt, circle] (a) at (1,0) [draw] {};

\node[inner sep=2pt, circle] (c) at (-1,0) [draw] {};
\node[inner sep=2pt, circle] (d) at (0,-1) [draw] {};
\node[inner sep=2pt, circle] (u) at (-1,1.5) [draw] {};
\node[inner sep=2pt, circle] (v) at (1,1.5) [draw] {};

\draw [-] (c) -- (d)node[midway, below] {$$};
\draw [-] (d) -- (a)node[midway, below] {$$};
\draw [-] (c) -- (u)node[midway, below] {$$};
\draw [-] (a) -- (v)node[midway, below] {$$};
\draw [-] (u) -- (v)node[midway, below] {$$};

\end{tikzpicture}
$
This shows that $S^1_4$ is simple digitally equivalent to $S^1_5$.  The sequence of transformations may be repeated three more times to obtain $S^1_8$.

\end{example}

Simple digital equivalence is an equivalence relation and numerical quantities, such as Euler characteristic and Betti numbers, are invariant under the relation \cite{Ivan1991}.   Notice that unlike the removal of a point, removal of an edge does not have any ``physical" interpretation for a digital image.  Hence we see the further need to work in the category of tolerance spaces. In addition, it is worth noting that simple digital equivalence for tolerance spaces (digital images) is different from digital homotopy equivalence used by Boxer \cite{Boxer2005}.  It is similar to that of simple homotopy theory for simplicial complexes proposed by Whitehead \cite{Whitehead1950}.\\

We will utilize the following helpful result, which shows that when determining if a digital image is simple digitally contractible, it suffices to remove only simple vertices.

\begin{theorem}\label{thm: conractible veretx}\cite[Thereom 3.8]{Evako94} A digital image $G$ is simple digitally contractible if and only if there is a sequence of deletions of vertices $v_1, v_2, \ldots, v_n$ transforming $G$ to $K_1$.
\end{theorem}


Now, from \cite{Evako2015} recall that

\begin{definition}\label{def: digital sphere} A digital $0$-sphere $S^0$ is a tolerance space consisting of two disconnected points.  Recursively, a connected graph $S^n$ is called a \emph{digital $n$-sphere}, $n>0$, if
\begin{itemize}
\item[1] For any $v\in S^n$, the link $\lk_G(v)$ is a digital $(n-1)$-sphere.
\item[2] For any $v\in S^n$, $S^n-v$ is simple digitally contractible.
\end{itemize}
\end{definition}

While we will use Evako's notion of simple point, we use a more general version of a digital $n$-sphere which we call a digital homology $n$-sphere. The homology $H_*(X)$ of a tolerance space $X$ is defined to be the integral homology of its corresponding clique complex.

\begin{definition}\label{def: digital homology sphere} A digital homology $0$-sphere $S^0$ is a tolerance space consisting of two disconnected simple digitally contractible components.  Recursively, a connected tolerance space $S^n$ is called a \emph{digital homology $n$-sphere}, $n>0$, if
\begin{itemize}
\item[1$^\prime$] for any $v\in S^n$, the link $\lk_{S^n}(v)$ is simple digitally equivalent to a digital $(n-1)$-sphere.
\item[2$^\prime$] $\widetilde{H}_i(S^n)=\Z$ for $i=n$ and $0$ otherwise.
\end{itemize}
\end{definition}

Because all digital $n$-spheres \cite{Matveev2007} have reduced homology $\Z$ in dimension $n$ and $0$ otherwise, we have the following.

\begin{proposition} If $S^n$ is a digital $n$-sphere, then $S^n$ is a digital homology $n$-sphere.
\end{proposition}

Section \ref{sec: Digital Hopf bundle} constructs a digital ``Hopf fibration" using a digital fiber bundle construction.  Because the total space is $S^1_8*S^1_8$ (see below), this digital space ought to somehow be a digital $3$-sphere.  Its links, however, do not satisfy condition 1) above, but they do satisfy condition 1$^\prime$).  Furthermore, because non-trivial examples of digital fibrations are few and far between in this vastly unexplored world of digital homotopy, it would seem appropriate to use this weaker definition of a digital sphere and explore its ramifications.  In other words, we will show that $S^1_8*S^1_8$ is a digital homology $3$-sphere.

\begin{example}\label{ex: sphere}
The tolerance space below satisfies both 2) and 1$^\prime$) with some links only simple digitally equivalent $S^1$ so that it does not satisfy 1). Note that edge $uv$ is a simple edge, and its removal results in a digital $2$-sphere which does satisfy 1).  Thus, we have an example of a tolerance space which satisfies 1$^\prime$) and 2) while not satisfying 1), but is simple digitally equivalent to a space that satisfies 1). In addition, it also satisfies 2$^\prime$) so that this tolerance space is a homology $2$-sphere.

$$
\begin{tikzpicture}[scale=.24,
    decoration={markings,mark=at position 0.6 with {\arrow{triangle 60}}},
    ]

\node[inner sep=2pt, circle] (1) at (8,0) [draw] {};
\node[inner sep=2pt, circle] (2) at (15,4) [draw] {};
\node[inner sep=2pt, circle] (3) at (1,10) [draw] {};
\node[inner sep=2pt, circle] (4) at (5,7) [draw] {};
\node[inner sep=2pt, circle] (5) at (12,7) [draw] {};
\node[inner sep=2pt, circle] (6) at (18,10) [draw] {};
\node[inner sep=2pt, circle] (7) at (10,12) [draw] {};
\node[inner sep=2pt, circle] (8) at (10,18) [draw] {};

\draw [-] (1) -- (3)node[midway, below] {$$};
\draw [-] (1) -- (4)node[midway, below] {$$};
\draw [-] (1) -- (5)node[midway, below] {$$};
\draw [-] (1) -- (2)node[midway, below] {$$};

\draw [-] (2) -- (5)node[midway, below] {$$};
\draw [dashed] (2) -- (3)node[midway, below] {$$};
\draw [dashed] (2) -- (7)node[midway, below] {$$};
\draw [-] (2) -- (6)node[midway, below] {$$};

\draw [-] (3) -- (8)node[midway, below] {$$};
\draw [-] (3) -- (4)node[midway, below] {$$};
\draw [dashed] (3) -- (7)node[midway, below] {$$};

\draw [-] (4) -- (5)node[midway, below] {$$};
\draw [dashed] (4) -- (7)node[midway, below] {$$};
\draw [-] (4) -- (8)node[midway, below] {$$};

\draw [-] (5) -- (8)node[midway, below] {$$};
\draw [-] (5) -- (6)node[midway, below] {$$};

\draw [dashed] (6) -- (7)node[midway, below] {$$};
\draw [-] (6) -- (8)node[midway, below] {$$};

\draw [dashed] (8) -- (7)node[midway, below] {$$};

\node[anchor = north east]  at (4) {\small{$u$}};
\node[anchor = north]  at (7) {\small{$v$}};

\end{tikzpicture}
$$

\end{example}

\begin{example}\label{exam:digcirc}
Recall from Example \ref{exam: circle hom} the digital $1$-sphere $S^1_4$.  This is also called the \emph{Diamond}, and it can be viewed in $\Z^2$ by taking the points
$(1,0),\ (0,1),\ (-1,0),\ (0,-1)$ along with $8$-adjacency. Using the definition of homotopy in \cite{Boxer2005}, it can be shown that it is digitally contractible. However, for our notion of
homotopy given in \defref{def: contractible}, this is no longer the case. For us, this is much more intuitive and reflects the non-trivial hole in the diamond.
\end{example}

%
%
%

It is well known that the classical Hopf fibration is not only a fiber bundle, but has the additional property of being principal; that is, there is a group action satisfying certain properties and interacting nicely with the continuity of the Hopf fibration.  Unfortunately the following result emphasizes the chasm that lies between topology and the digital world. By a digital group $G$, we mean a tolerance space $G$ with tolerance $\sim$ and a continuous group multiplication $\cdot$ with continuous inverse on $G$ in the sense that if $a\sim b$ and $c\sim d$ then $a\cdot c\sim b\cdot d$. For any digital group $G$ and $g\in G$, the \emph{star} of $g$ in $G$, denoted $\mathrm{st}(g)$ is the induced subgraph on $g\cup \lk(g)$.

\begin{theorem}\label{thm no tolerance group} Suppose $G$ is a finite, connected digital group.   Then $G$ is simple digitally contractible.
\end{theorem}

\begin{proof} We prove the slightly stronger result that $G$ is in fact a complete graph.   Let $e\in G$ be the identity.  We first show that $\mathrm{st}(e)$ is a complete graph and a subgroup.  Let $a,b\in \mathrm{st}(e)$.  Then $a\sim e$ and $b\sim e$, so both $ab\sim e$ and $a\sim b$.  This shows that $\mathrm{st}(e)$ is both a complete graph and closed under multiplication.  Furthermore, if $a\in \mathrm{st}(e)$, then $a\sim e$ and $a\sim e$ so that $a^2\sim e$.  Inductively, $a^k\sim e$ for all powers $k$ and since $G$ is finite, $a^{-1}\sim e$. Hence $\mathrm{st}(e)$ is a subgroup. Now $\mathrm{st}(x)$ and $\mathrm{st}(e)$ are isomorphic as groups for every $x\in G$ since $x\,\mathrm{st}(e)=\mathrm{st}(x)$ and left multiplication by a fixed $x$ on a fixed group is a group isomorphism.  Thus the star of every element of $G$ is the complete graph $K_n$ for some $n$.  But since $G$ is connected,
this means that, in fact, $G$ must equal $K_n$.  Indeed, if $b \in \mathrm{lk}(e)$, then $b$ is adjacent to $e$ and the other $n-2$ elements in the link of $e$. But $\mathrm{st}(b)=
K_n$ and $b$ is already adjacent to $n-1$ points, so it cannot be adjacent to any other point in $G$. Since this holds for all $b \in \mathrm{lk}(e)$, we see that
$\mathrm{st}(e)=K_n$ is not connected to any other point of $G$. This contradicts the connectedness of $G$ unless $G=\mathrm{st}(e)=K_n$. We conclude that
$G$ is a complete graph and by Example \ref{ex: complete graph}, simple digitally contractible.
\end{proof}

\section{Cones, Suspensions, and Joins}\label{sec:cones}
Let $X$ be a tolerance space. We now define the straightforward analogues of basic (and essential) homotopy constructs.

\begin{definition}\label{def:conesuspjoin}
(1.) The \emph{cone} on $X$, denoted $CX$, is given by $CX = X \cup \{A\}$ where $A$ is taken to be adjacent to all $x \in X$. We refer to $A$ as the \emph{apex} of $CX$.

(2.) The \emph{suspension} on $X$, denoted $SX$, is given by $SX = X \cup\{A,B\}$ with $A\sim x$, $B \sim x$ for all $x \in X$
and $A \not\sim B$.

(3.) If $Y$ is another set with adjacency, then the \emph{join} of $X$ and $Y$ , denoted $X\ast Y$, is given by
$$X \ast Y = X \times CY \cup CX \times Y.$$
The adjacency relation comes from the product adjacency relation in $X \times CY$ and $CX\times Y$.

\end{definition}

From Proposition \ref{prop: tolerance is digital image} it follows that there is a digital cone and suspension, at least as long as we allow (re-)embedding the digital image in a higher-dimensional ambient $\Z^{n}$.

\begin{corollary}
If $X$ is a digital image, we may embed $X$ in $\Z^{n}$ and form a cone or suspension in $\Z^{n}$, where $n = |X|$, the number of vertices.
\end{corollary}

\begin{proof}
Embed $X$ in $[-1, 1]^{n-1} \subseteq \Z^{n-1} \subseteq \Z^{n-1} \times \{0\} \subseteq \Z^{n}$ as above.  For $CX$, take the apex  $\{A\} = \textbf{e}_n = (0, \ldots, 0, 1)$, which is adjacent to every point in  $[-1, 1]^{n-1} \times \{0\} \subseteq \Z^n$, and hence to every point of $X$ as we have embedded it.  For $SX$, form upper and lower cones with   $\{A\} = \textbf{e}_n$ and $\{B\} = -\textbf{e}_n$.  Note that these $\{A\}$ and $\{B\}$  are not adjacent to each other.
\end{proof}

\begin{proposition}\label{prop: cone contractible} The cone $CX$ on any tolerance space $X$ is simple digitally contractible.
\end{proposition}

\begin{proof} We proceed by induction on the number of vertices in a cone $CX$.  If $X$ is a one point digital image, clearly $CX$ is simple digitally contractible. Now assume that any cone on $k\geq 1$ vertices is simple digitally contractible, and let $X$ be a tolerance space with $k+1$ vertices. Consider any point $v\in CX$, $v\neq A$.  Now $A\in \lk_{CX}(v)$ and furthermore, $A$ is adjacent to all the neighbors of $v$, so $\lk_{CX}(v)$ is itself a cone on the neighbors of $v$ with apex $A$.  By the inductive hypothesis, the link of $v$ is simple digitally contractible, so $v$ is a simple point and may be removed from $CX$, yielding $C(X-\{v\})$, which itself is a cone on $k$ vertices.  Again by the inductive hypothesis, $C(X-\{v\})$ is simple digitally contractible.  Thus $CX$ is simple digitally contractible and the result is shown.
\end{proof}

\begin{remark}\label{rem:conesuspjoin}
Denote the cone on $X$ as above and the cone on $Y$ as $CY=Y \cup \{B\}$. Then the join of $X$ and $Y$ can be pictured as having
three levels: $X \times \{B\}$ on top, $X \times Y$ in the middle and $\{A\} \times Y$ on the bottom. In particular, a point $(x,B)$ on top
is adjacent only to all  $(x',y) \in X \times Y$ where $x\sim x' \in X$ (since $B\sim y$ for all $y \in Y$). Of course, $(x,B)\sim (x',B)$ for
these same $x'\in X$ as well. Similarly we can identify adjacencies between the middle and bottom level.
\end{remark}

\begin{example}\label{exam:spheres}
If $X=S^1_4=S^1$ as in Example \ref{exam:digcirc}, then $CS^1 = S^1 \cup \{A\}$, so we obtain the top half of an octahedron. This is the digital $2$-disk $D^2 = CS^1$, which is isomorphic to the Diamond $D$ with the hole at $(0,0)$ filled in. The suspension on $S^1$ is given by $SS^1 = S^1 \cup \{A,B\}$
which is simply two $2$-disks sewn together along their boundaries. This then is the digital $2$-sphere $S^2$ (with $6$ points).
Continuing in this fashion produces spheres $S^n$ with $2n+2$ points.
\end{example}

\begin{example}\label{exam: S2 with any equator} We illustrate how $S^2$ defined in Example \ref{exam:spheres} is simple digitally equivalent to the suspension on any cycle of length greater than 3. Starting with $S^2$, we attach a simple point and remove a simple edge as below:

$
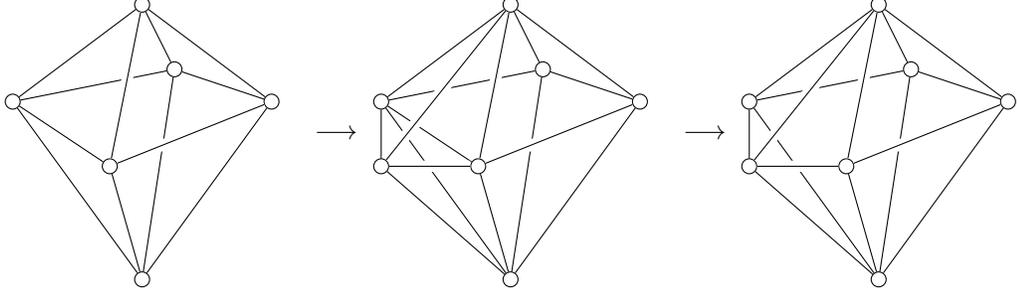
\begin{tikzpicture}[scale=.43,
    decoration={markings,mark=at position 0.6 with {\arrow{triangle 60}}},
    ]

\node[inner sep=2pt, circle] (a) at (3,0) [draw] {};
\node[inner sep=2pt, circle] (b) at (8,2) [draw] {};
\node[inner sep=2pt, circle] (c) at (5,3) [draw] {};
\node[inner sep=2pt, circle] (d) at (0,2) [draw] {};
\node[inner sep=2pt, circle] (N) at (4,5) [draw] {};
\node[inner sep=2pt, circle] (S) at (4,-3.5) [draw] {};

\draw [-] (c) -- (d)node[midway, below] {$$};
\draw [-] (d) -- (a)node[midway, below] {$$};

\draw [-] (N) -- (b)node[midway, below] {$$};
\draw [-] (N) -- (c)node[midway, below] {$$};
\draw [-] (N) -- (d)node[midway, below] {$$};
\draw [-] (S) -- (a)node[midway, below] {$$};

\draw [-] (S) -- (c)node[midway, below] {$$};
\draw [-] (S) -- (d)node[midway, below] {$$};
\draw [white, line width=1.5mm] (a) -- (b) node[midway, left] {$$};
\draw [-] (a) -- (b)node[midway, below] {$$};
\draw [white, line width=1.5mm] (N) -- (a) node[midway, left] {$$};
\draw [-] (N) -- (a)node[midway, below] {$$};
\draw [-] (S) -- (b)node[midway, below] {$$};
\draw [-] (b) -- (c)node[midway, below] {$$};

\node at (10,1) {$\longrightarrow$};
\end{tikzpicture}
$
$
\begin{tikzpicture}[scale=.43,
    decoration={markings,mark=at position 0.6 with {\arrow{triangle 60}}},
    ]

\node[inner sep=2pt, circle] (a) at (3,0) [draw] {};
\node[inner sep=2pt, circle] (b) at (8,2) [draw] {};
\node[inner sep=2pt, circle] (c) at (5,3) [draw] {};
\node[inner sep=2pt, circle] (d) at (0,2) [draw] {};
\node[inner sep=2pt, circle] (N) at (4,5) [draw] {};
\node[inner sep=2pt, circle] (S) at (4,-3.5) [draw] {};
\node[inner sep=2pt, circle] (x) at (0,0) [draw] {};

\draw [-] (c) -- (d)node[midway, below] {$$};
\draw [-] (d) -- (a)node[midway, below] {$$};

\draw [-] (N) -- (b)node[midway, below] {$$};
\draw [-] (N) -- (c)node[midway, below] {$$};

\draw [-] (S) -- (a)node[midway, below] {$$};

\draw [-] (S) -- (c)node[midway, below] {$$};
\draw [-] (S) -- (d)node[midway, below] {$$};
\draw [-] (x) -- (d)node[midway, below] {$$};
\draw [white, line width=1.5mm] (x) -- (a) node[midway, left] {$$};
\draw [-] (x) -- (a)node[midway, below] {$$};

\draw [-] (x) -- (S)node[midway, below] {$$};
\draw [white, line width=1.5mm] (a) -- (b) node[midway, left] {$$};
\draw [-] (a) -- (b)node[midway, below] {$$};
\draw [white, line width=1.5mm] (N) -- (a) node[midway, left] {$$};

\draw [white, line width=1.5mm] (x) -- (N) node[midway, left] {$$};
\draw [-] (x) -- (N)node[midway, below] {$$};
\draw [-] (N) -- (d)node[midway, below] {$$};
\draw [-] (S) -- (b)node[midway, below] {$$};
\draw [-] (b) -- (c)node[midway, below] {$$};
\draw [-] (N) -- (a)node[midway, below] {$$};

\node at (10,1) {$\longrightarrow$};
\end{tikzpicture}
$
$
\begin{tikzpicture}[scale=.43,
    decoration={markings,mark=at position 0.6 with {\arrow{triangle 60}}},
    ]

\node[inner sep=2pt, circle] (a) at (3,0) [draw] {};
\node[inner sep=2pt, circle] (b) at (8,2) [draw] {};
\node[inner sep=2pt, circle] (c) at (5,3) [draw] {};
\node[inner sep=2pt, circle] (d) at (0,2) [draw] {};
\node[inner sep=2pt, circle] (N) at (4,5) [draw] {};
\node[inner sep=2pt, circle] (S) at (4,-3.5) [draw] {};
\node[inner sep=2pt, circle] (x) at (0,0) [draw] {};

\draw [-] (c) -- (d)node[midway, below] {$$};

\draw [-] (N) -- (b)node[midway, below] {$$};
\draw [-] (N) -- (c)node[midway, below] {$$};

\draw [-] (S) -- (a)node[midway, below] {$$};

\draw [-] (S) -- (c)node[midway, below] {$$};
\draw [-] (S) -- (d)node[midway, below] {$$};
\draw [-] (x) -- (d)node[midway, below] {$$};
\draw [white, line width=1.5mm] (x) -- (a) node[midway, left] {$$};
\draw [-] (x) -- (a)node[midway, below] {$$};

\draw [-] (x) -- (S)node[midway, below] {$$};
\draw [white, line width=1.5mm] (a) -- (b) node[midway, left] {$$};
\draw [-] (a) -- (b)node[midway, below] {$$};
\draw [white, line width=1.5mm] (N) -- (a) node[midway, left] {$$};

\draw [white, line width=1.5mm] (x) -- (N) node[midway, left] {$$};
\draw [-] (x) -- (N)node[midway, below] {$$};
\draw [-] (N) -- (d)node[midway, below] {$$};
\draw [-] (S) -- (b)node[midway, below] {$$};
\draw [-] (b) -- (c)node[midway, below] {$$};
\draw [-] (N) -- (a)node[midway, below] {$$};

\end{tikzpicture}
$

Continuing in this manner, we obtain the suspension on a cycle of any length.

\end{example}
%
%

\section{A digital homology $3$-sphere}

We devote this section to constructing what we view as a representation of $S^3$ in the tolerance space setting.  We give this as a join of circles because subsequently we wish to mimic the Hopf construction to obtain an analogue of the Hopf map on this $S^3$.  Now this construction in the topological setting is based upon having a multiplication on the circle.  As we shall see, however, there is no such multiplication in the tolerance setting (Proposition \ref{prop: no circle multiplication}).  It is this choice that drives our choice of $S^1_8$ as the circle we use to construct our $S^3$.

Let $S^1_8$ be an eight point circle, and consider two copies $X,Y$ of this digital space with labelings given by

{
\begin{center}
$
\begin{tikzpicture}[scale=.5,
    decoration={markings,mark=at position 0.6 with {\arrow{triangle 60}}},
    ]

\node[inner sep=2pt, circle] (1) at (4,3.5) [draw] {};
\node[inner sep=2pt, circle] (2) at (2,5) [draw] {};
\node[inner sep=2pt, circle] (3) at (0,5) [draw] {};
\node[inner sep=2pt, circle] (4) at (-2,3.5) [draw] {};
\node[inner sep=2pt, circle] (5) at (-2,1.5) [draw] {};
\node[inner sep=2pt, circle] (6) at (0,0) [draw] {};
\node[inner sep=2pt, circle] (7) at (2,0) [draw] {};
\node[inner sep=2pt, circle] (8) at (4,1.5) [draw] {};

\draw [-] (1) -- (2)node[midway, below] {$$};
\draw [-] (2) -- (3)node[midway, below] {$$};
\draw [-] (3) -- (4)node[midway, below] {$$};
\draw [-] (4) -- (5)node[midway, below] {$$};
\draw [-] (5) -- (6)node[midway, below] {$$};
\draw [-] (7) -- (6)node[midway, below] {$$};
\draw [-] (7) -- (8)node[midway, below] {$$};
\draw [-] (8) -- (1)node[midway, below] {$$};

\node[anchor = west]  at (1) {\small{$x_1$}};
\node[anchor = south]  at (2) {\small{$x_2$}};
\node[anchor = south]  at (3) {\small{$x_3$}};
\node[anchor = east]  at (4) {\small{$x_4$}};
\node[anchor = east]  at (5) {\small{$x_5$}};
\node[anchor = north]  at (6) {\small{$x_6$}};
\node[anchor = north]  at (7) {\small{$x_7$}};
\node[anchor = west]  at (8) {\small{$x_8$}};.
\end{tikzpicture}
$
\hfill
$
\begin{tikzpicture}[scale=.5,
    decoration={markings,mark=at position 0.6 with {\arrow{triangle 60}}},
    ]

\node[inner sep=2pt, circle] (1) at (4,3.5) [draw] {};
\node[inner sep=2pt, circle] (2) at (2,5) [draw] {};
\node[inner sep=2pt, circle] (3) at (0,5) [draw] {};
\node[inner sep=2pt, circle] (4) at (-2,3.5) [draw] {};
\node[inner sep=2pt, circle] (5) at (-2,1.5) [draw] {};
\node[inner sep=2pt, circle] (6) at (0,0) [draw] {};
\node[inner sep=2pt, circle] (7) at (2,0) [draw] {};
\node[inner sep=2pt, circle] (8) at (4,1.5) [draw] {};

\draw [-] (1) -- (2)node[midway, below] {$$};
\draw [-] (2) -- (3)node[midway, below] {$$};
\draw [-] (3) -- (4)node[midway, below] {$$};
\draw [-] (4) -- (5)node[midway, below] {$$};
\draw [-] (5) -- (6)node[midway, below] {$$};
\draw [-] (7) -- (6)node[midway, below] {$$};
\draw [-] (7) -- (8)node[midway, below] {$$};
\draw [-] (8) -- (1)node[midway, below] {$$};

\node[anchor = west]  at (1) {\small{$y_1$}};
\node[anchor = south]  at (2) {\small{$y_2$}};
\node[anchor = south]  at (3) {\small{$y_3$}};
\node[anchor = east]  at (4) {\small{$y_4$}};
\node[anchor = east]  at (5) {\small{$y_5$}};
\node[anchor = north]  at (6) {\small{$y_6$}};
\node[anchor = north]  at (7) {\small{$y_7$}};
\node[anchor = west]  at (8) {\small{$y_8$}};.
\end{tikzpicture}
$
\end{center}}

\begin{theorem}\label{thm: join of S1 is S3} Let $X\ast Y=S^1_8\ast S^1_8=CS^1_8\times S^1_8\cup S^1_8\times CS^1_8=\{A\cup S^1_8\}\times S^1_8\cup S^1_8\times \{B\cup S_8^1\}$. Then $X\ast Y$ is a digital homology $3$-sphere.
\end{theorem}

\begin{proof}
According to Definition \ref{def: digital homology sphere}, we must check two conditions. First we must check condition 1$^\prime$ of Definition \ref{def: digital homology sphere}, i.e., we need to show that the link of each vertex is simple digitally equivalent to a digital $2$-sphere. We claim that there are two link isomorphism types in $S^1_8\ast S^1_8$, which we call Type I and Type II. Type I links are of the form $\lk((A,y_j))$ and $\lk((x_i,B))$.  Indeed, let $v:=(A,y_j)$.  Then the vertex set of $\lk(v)$ is given by $S^1_8\times \{y_{j-1},y_j,y_{j+1}, B\}\cup \{(A,y_{j-1}),(A,y_{j+1})\}$, with subscripts reduced mod 8, so that $|\lk(v)|=32$.  We first claim that each of the points $(x_i,B)$ is a simple point of $\lk(v)$ and may furthermore be removed in any order.  To see this, observe that the link of $(x_i,B)$ in $\lk(v)$ is given by

$
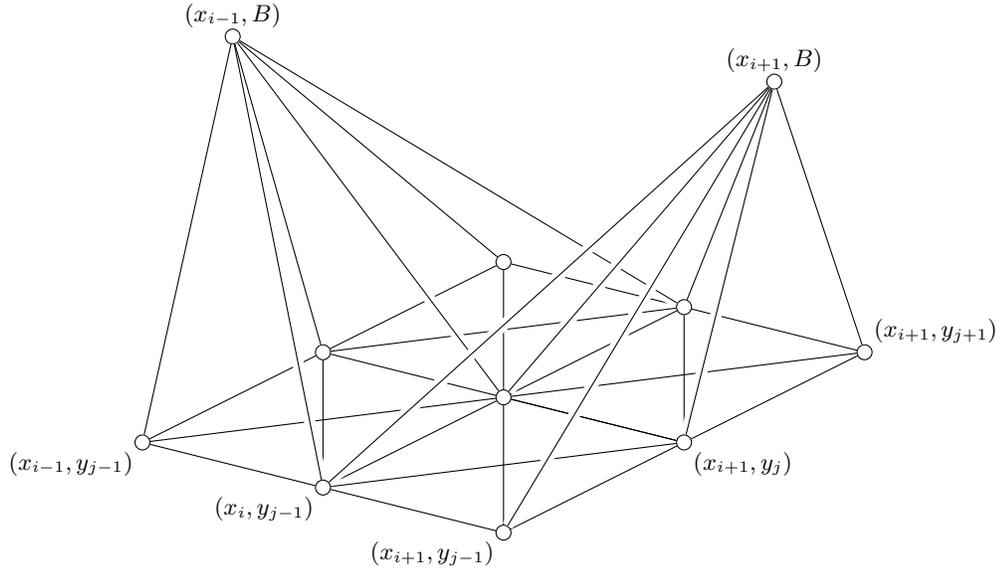
\begin{tikzpicture}[scale=.6,
    decoration={markings,mark=at position 0.6 with {\arrow{triangle 60}}},
    ]

\node[inner sep=2pt, circle] (10) at (0,2) [draw] {};
\node[inner sep=2pt, circle] (20) at (4,1) [draw] {};
\node[inner sep=2pt, circle] (30) at (8,0) [draw] {};
\node[inner sep=2pt, circle] (11) at (4,4) [draw] {};
\node[inner sep=2pt, circle] (21) at (8,3) [draw] {};
\node[inner sep=2pt, circle] (31) at (12,2) [draw] {};
\node[inner sep=2pt, circle] (12) at (8,6) [draw] {};
\node[inner sep=2pt, circle] (22) at (12,5) [draw] {};
\node[inner sep=2pt, circle] (32) at (16,4) [draw] {};

\node[inner sep=2pt, circle] (i-1) at (2,11) [draw] {};
\node[inner sep=2pt, circle] (i+1) at (14,10) [draw] {};

\draw [-] (10) -- (11)node[midway, below] {$$};
\draw [-] (10) -- (20)node[midway, below] {$$};
\draw [-] (10) -- (21)node[midway, below] {$$};

\draw [-] (20) -- (21)node[midway, below] {$$};
\draw [-] (20) -- (30)node[midway, below] {$$};
\draw [-] (20) -- (31)node[midway, below] {$$};
\draw [-] (20) -- (11)node[midway, below] {$$};

\draw [-] (11) -- (12)node[midway, below] {$$};
\draw [-] (11) -- (22)node[midway, below] {$$};
\draw [-] (11) -- (21)node[midway, below] {$$};

\draw [-] (12) -- (22)node[midway, below] {$$};
\draw [-] (12) -- (21)node[midway, below] {$$};

\draw [-] (21) -- (22)node[midway, below] {$$};
\draw [-] (21) -- (30)node[midway, below] {$$};
\draw [-] (21) -- (31)node[midway, below] {$$};
\draw [-] (30) -- (31)node[midway, below] {$$};
\draw [-] (21) -- (32)node[midway, below] {$$};
\draw [-] (21) -- (31)node[midway, below] {$$};

\draw [-] (32) -- (22)node[midway, below] {$$};
\draw [-] (31) -- (22)node[midway, below] {$$};
\draw [-] (31) -- (32)node[midway, below] {$$};

\draw [white, line width=1.5mm] (11) -- (i-1) node[midway, left] {$$};
\draw [white, line width=1.5mm] (12) -- (i-1) node[midway, left] {$$};
\draw [white, line width=1.5mm] (10) -- (i-1) node[midway, left] {$$};
\draw [white, line width=1.5mm] (21) -- (i-1) node[midway, left] {$$};
\draw [white, line width=1.5mm] (22) -- (i-1) node[midway, left] {$$};
\draw [white, line width=1.5mm] (20) -- (i-1) node[midway, left] {$$};
\draw [-] (11) -- (i-1)node[midway, below] {$$};
\draw [-] (12) -- (i-1)node[midway, below] {$$};
\draw [-] (10) -- (i-1)node[midway, below] {$$};
\draw [-] (21) -- (i-1)node[midway, below] {$$};
\draw [-] (22) -- (i-1)node[midway, below] {$$};
\draw [-] (20) -- (i-1)node[midway, below] {$$};

\draw [white, line width=1.5mm] (20) -- (i+1) node[midway, left] {$$};
\draw [white, line width=1.5mm] (21) -- (i+1) node[midway, left] {$$};
\draw [white, line width=1.5mm] (22) -- (i+1) node[midway, left] {$$};
\draw [white, line width=1.5mm] (30) -- (i+1) node[midway, left] {$$};
\draw [white, line width=1.5mm] (31) -- (i+1) node[midway, left] {$$};
\draw [white, line width=1.5mm] (32) -- (i+1) node[midway, left] {$$};
\draw [-] (21) -- (i+1)node[midway, below] {$$};
\draw [-] (22) -- (i+1)node[midway, below] {$$};
\draw [-] (20) -- (i+1)node[midway, below] {$$};
\draw [-] (31) -- (i+1)node[midway, below] {$$};
\draw [-] (32) -- (i+1)node[midway, below] {$$};
\draw [-] (30) -- (i+1)node[midway, below] {$$};

\node[anchor = south west]  at (32) {\small{$(x_{i+1},y_{j+1})$}};
\node[anchor = north east]  at (10) {\small{$(x_{i-1},y_{j-1})$}};
\node[anchor = north east]  at (20) {\small{$(x_{i},y_{j-1})$}};
\node[anchor = north east]  at (30) {\small{$(x_{i+1},y_{j-1})$}};
\node[anchor = north west]  at (31) {\small{$(x_{i+1},y_{j})$}};
\node[anchor = south]  at (i-1) {\small{$(x_{i-1},B)$}};
\node[anchor = south]  at (i+1) {\small{$(x_{i+1},B)$}};
.
\end{tikzpicture}
$

It is clear that this is simple digitally contractible.  It is also clear that the resulting space remains simple digitally contractible with $(x_{i-1},B)$ or $(x_{i+1},B)$ removed.  Next, after removing all the $(x_i,B)$ from $\lk(v)$, we claim that each of the points in $X\times \{y_{j-1}\}$ and $X\times \{y_{j+1}\}$ is simple and removable in any order.  To see this, observe that for any $i=1,\ldots, 8$, the point $(A,y_{j-1})$ is adjacent to every element in $\lk(x_i,y_{j-1})$; that is, $\lk(x_i,y_{j-1})$ is a cone with apex $(A,y_{j-1})$.  By Proposition \ref{prop: cone contractible}, these links are simple digitally contractible, and they furthermore remain cones when any point of $X\times \{y_{j-1}\}$ is removed.  Thus every point of $X\times \{y_{j-1}\}$ is a simple point.  In a similar way, the link of every point in $X\times \{y_{j+1}\}$ is a cone with apex $(A,y_{j+1})$.

Thus after removing all of these simple points, we obtain

$$
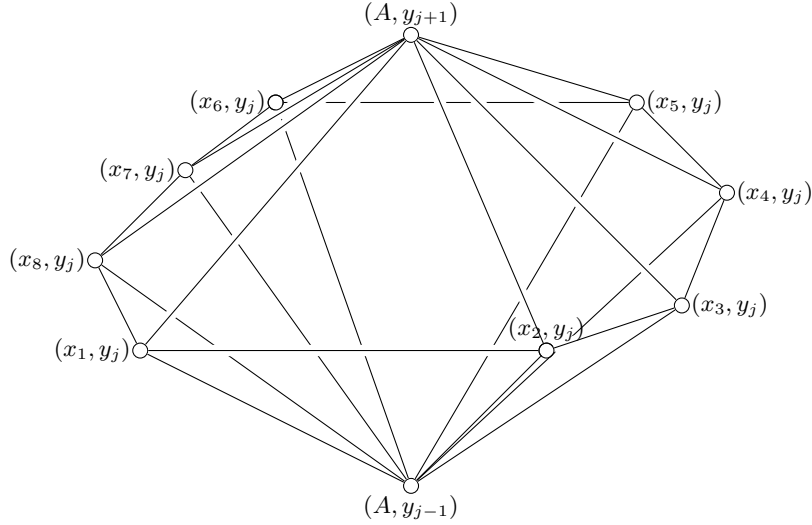
\begin{tikzpicture}[scale=.6,
    decoration={markings,mark=at position 0.6 with {\arrow{triangle 60}}},
    ]

\node[inner sep=2pt, circle] (51) at (12,15.5) [draw] {};
\node[inner sep=2pt, circle] (71) at (2,14) [draw] {};
\node[inner sep=2pt, circle] (11) at (1,10) [draw] {};
\node[inner sep=2pt, circle] (31) at (13,11) [draw] {};
\node[inner sep=2pt, circle] (41) at (14,13.5) [draw] {};
\node[inner sep=2pt, circle] (81) at (0,12) [draw] {};
\node[inner sep=2pt, circle] (61) at (4,15.5) [draw] {};
\node[inner sep=2pt, circle] (21) at (10,10) [draw] {};

\node[inner sep=2pt, circle] (A0) at (7,7) [draw] {};
\node[inner sep=2pt, circle] (A1) at (7,17) [draw] {};


\draw [white, line width=1.5mm] (11) -- (A0) node[midway, left] {$$};
\draw [white, line width=1.5mm] (21) -- (A0) node[midway, left] {$$};
\draw [white, line width=1.5mm] (31) -- (A0) node[midway, left] {$$};
\draw [white, line width=1.5mm] (41) -- (A0) node[midway, left] {$$};
\draw [white, line width=1.5mm] (51) -- (A0) node[midway, left] {$$};
\draw [white, line width=1.5mm] (61) -- (A0) node[midway, left] {$$};
\draw [white, line width=1.5mm] (71) -- (A0) node[midway, left] {$$};
\draw [white, line width=1.5mm] (81) -- (A0) node[midway, left] {$$};

\draw [-] (51) -- (A0) node[midway, left] {$$};
\draw [-] (61) -- (A0)node[midway, below] {$$};
\draw [-] (21) -- (A0)node[midway, below] {$$};
\draw [-] (41) -- (A0)node[midway, below] {$$};
\draw [-] (71) -- (A0)node[midway, below] {$$};
\draw [-] (81) -- (A0)node[midway, below] {$$};

\draw [-] (51) -- (61) node[midway, left] {$$};

\draw [white, line width=1.5mm] (11) -- (A1) node[midway, left] {$$};
\draw [white, line width=1.5mm] (21) -- (A1) node[midway, left] {$$};
\draw [white, line width=1.5mm] (31) -- (A1) node[midway, left] {$$};
\draw [white, line width=1.5mm] (41) -- (A1) node[midway, left] {$$};
\draw [white, line width=1.5mm] (71) -- (A1) node[midway, left] {$$};
\draw [white, line width=1.5mm] (81) -- (A1) node[midway, left] {$$};

\draw [-] (11) -- (A1) node[midway, left] {$$};
\draw [-] (31) -- (A1)node[midway, below] {$$};
\draw [-] (41) -- (A1)node[midway, below] {$$};
\draw [-] (51) -- (A1) node[midway, left] {$$};
\draw [-] (61) -- (A1)node[midway, below] {$$};
\draw [-] (71) -- (A1)node[midway, below] {$$};
\draw [-] (81) -- (A1)node[midway, below] {$$};


\draw [-] (41) -- (51)node[midway, below] {$$};
\draw [-] (61) -- (71)node[midway, below] {$$};

\draw [white, line width=1.5mm] (21) -- (31) node[midway, left] {$$};
\draw [-] (31) -- (A0)node[midway, below] {$$};

\draw [-] (21) -- (31)node[midway, below] {$$};


%

\draw [-] (21) -- (A1)node[midway, below] {$$};

\draw [white, line width=1.5mm] (11) -- (21) node[midway, left] {$$};
\draw [-] (11) -- (21) node[midway, left] {$$};

\draw [-] (71) -- (81) node[midway, left] {$$};

\draw [-] (81) -- (11)node[midway, below] {$$};
\draw [-] (31) -- (41)node[midway, below] {$$};

\draw [-] (11) -- (A0) node[midway, left] {$$};

\node[anchor = south ]  at (A1) {\small{$(A,y_{j+1})$}};
\node[anchor = north ]  at (A0) {\small{$(A,y_{j-1})$}};

\node[anchor = east ]  at (11) {\small{$(x_1,y_j)$}};
\node[anchor = south ]  at (21) {\small{$(x_2,y_j)$}};
\node[anchor = west ]  at (31) {\small{$(x_3,y_j)$}};
\node[anchor = west ]  at (41) {\small{$(x_4,y_j)$}};
\node[anchor = west ]  at (51) {\small{$(x_5,y_j)$}};
\node[anchor = east ]  at (61) {\small{$(x_6,y_j)$}};
\node[anchor = east ]  at (71) {\small{$(x_7,y_j)$}};
\node[anchor = east ]  at (81) {\small{$(x_8,y_j)$}};

\node[inner sep=2pt, circle] (61') at (4,15.5) [draw] {};
\node[inner sep=2pt, circle] (21') at (10,10) [draw] {};
.
\end{tikzpicture}
$$
which is a digital $2$-sphere by Example \ref{exam: S2 with any equator}. The other case of a Type I link in which $v=(x_i,B)$ is similarly checked to be simple digitally equivalent to a $2$-sphere.  We omit the details.

Next we must check all Type II links i.e. links of the form $\lk(x_i,y_j)$. As a set, this is given by $\lk(x_i,y_j)=\{x_{i-1},x_i,x_{i+1}, A\}\times  \{y_{j-1},y_j,y_{j+1}, B\}-\{(x_i,y_i),(A,B)\}$ so that $|\lk(x_i,y_j)|=14$ and where we again reduce $i$ and $j$ modulo $8$. We observe that each of the points $(A,y_{j-1}),(A,y_{j+1}),(x_{i-1},B),$ and $(x_{i+1},B)$ are simple points and that their links are isomorphic.  The link of $(A,y_{j-1})$, for example, is given by

$$
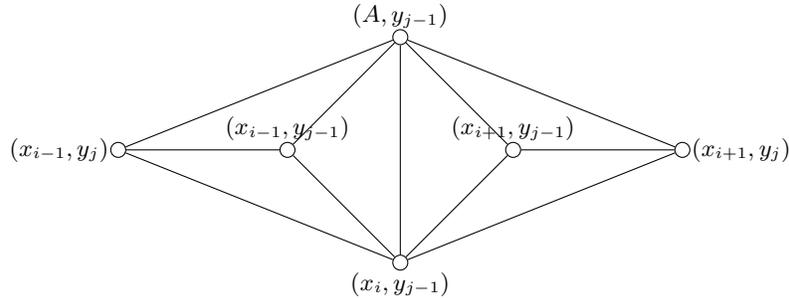
\begin{tikzpicture}[scale=.75,
    decoration={markings,mark=at position 0.6 with {\arrow{triangle 60}}},
    ]

\node[inner sep=2pt, circle] (8) at (0,0) [draw] {};
\node[inner sep=2pt, circle] (2) at (0,4) [draw] {};
\node[inner sep=2pt, circle] (9) at (2,2) [draw] {};
\node[inner sep=2pt, circle] (11) at (5,2) [draw] {};
\node[inner sep=2pt, circle] (7) at (-2,2) [draw] {};
\node[inner sep=2pt, circle] (10) at (-5,2) [draw] {};

\draw [-] (8) -- (2)node[midway, below] {$$};
\draw [-] (8) -- (7)node[midway, below] {$$};
\draw [-] (8) -- (10)node[midway, below] {$$};
\draw [-] (8) -- (9)node[midway, below] {$$};
\draw [-] (8) -- (11)node[midway, below] {$$};
\draw [-] (10) -- (7)node[midway, below] {$$};
\draw [-] (2) -- (7)node[midway, below] {$$};
\draw [-] (10) -- (2)node[midway, below] {$$};
\draw [-] (9) -- (11)node[midway, below] {$$};
\draw [-] (9) -- (2)node[midway, below] {$$};
\draw [-] (2) -- (11)node[midway, below] {$$};

\node[anchor = south ]  at (2) {\small{$(A,y_{j-1})$}};
\node[anchor = east ]  at (10) {\small{$(x_{i-1},y_{j})$}};
\node[anchor = south ]  at (7) {\small{$(x_{i-1},y_{j-1})$}};
\node[anchor = south ]  at (9) {\small{$(x_{i+1},y_{j-1})$}};
\node[anchor = north ]  at (8) {\small{$(x_{i},y_{j-1})$}};
\node[anchor = west ]  at (11) {\small{$(x_{i+1},y_{j})$}};
.
\end{tikzpicture}
$$
Furthermore, none of the four points $(A,y_{j-1}),(A,y_{j+1}),(x_{i-1},B),$ and $(x_{i+1},B)$ are in the link of any of the other points so that the four points may be removed in any order.  Removal of these points from $\lk(x_i,y_j)$ then yields
$$
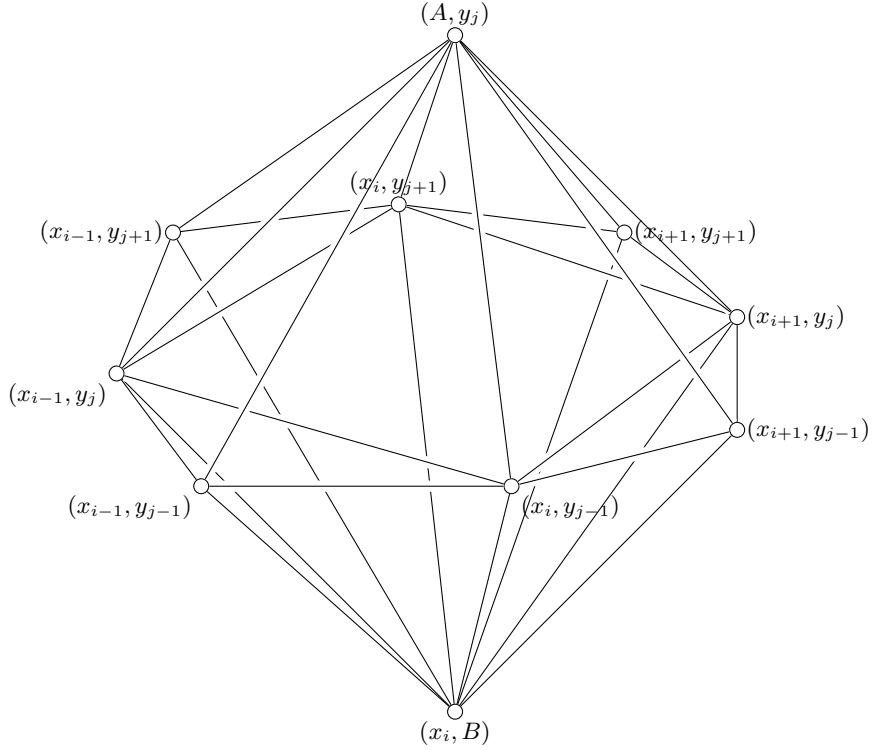
\begin{tikzpicture}[scale=.75,
    decoration={markings,mark=at position 0.6 with {\arrow{triangle 60}}},
    ]

\node[inner sep=2pt, circle] (1) at (-1.5,1) [draw] {};
\node[inner sep=2pt, circle] (2) at (4,1) [draw] {};
\node[inner sep=2pt, circle] (3) at (8,2) [draw] {};
\node[inner sep=2pt, circle] (4) at (8,4) [draw] {};
\node[inner sep=2pt, circle] (5) at (6,5.5) [draw] {};
\node[inner sep=2pt, circle] (6) at (2,6) [draw] {};
\node[inner sep=2pt, circle] (7) at (-2,5.5) [draw] {};
\node[inner sep=2pt, circle] (8) at (-3,3) [draw] {};

\node[inner sep=2pt, circle] (A) at (3,9) [draw] {};
\node[inner sep=2pt, circle] (B) at (3,-3) [draw] {};

\draw [-] (B) -- (1)node[midway, below] {$$};
\draw [-] (B) -- (2)node[midway, below] {$$};
\draw [-] (B) -- (5)node[midway, below] {$$};
\draw [-] (B) -- (6)node[midway, below] {$$};
\draw [-] (B) -- (7)node[midway, below] {$$};
\draw [-] (B) -- (8)node[midway, below] {$$};


\draw [white, line width=1.5mm] (2) -- (4)node[midway, below] {$$};
\draw [-] (B) -- (4)node[midway, below] {$$};
\draw [white, line width=1.5mm] (4) -- (6)node[midway, below] {$$};
\draw [white, line width=1.5mm] (8) -- (2)node[midway, below] {$$};
\draw [white, line width=1.5mm] (2) -- (3)node[midway, below] {$$};
\draw [white, line width=1.5mm] (4) -- (3)node[midway, below] {$$};
\draw [white, line width=1.5mm] (5) -- (4)node[midway, below] {$$};
\draw [white, line width=1.5mm] (6) -- (5)node[midway, below] {$$};
\draw [white, line width=1.5mm] (6) -- (7)node[midway, below] {$$};
\draw [white, line width=1.5mm] (8) -- (7)node[midway, below] {$$};
\draw [white, line width=1.5mm] (8) -- (6)node[midway, below] {$$};
\draw [white, line width=1.5mm] (1) -- (2) node[midway, left] {$$};
\draw [-] (1) -- (2)node[midway, below] {$$};
\draw [-] (1) -- (8)node[midway, below] {$$};
\draw [-] (8) -- (2)node[midway, below] {$$};
\draw [-] (2) -- (3)node[midway, below] {$$};
\draw [-] (4) -- (3)node[midway, below] {$$};
\draw [-] (2) -- (4)node[midway, below] {$$};
\draw [-] (4) -- (6)node[midway, below] {$$};
\draw [-] (6) -- (5)node[midway, below] {$$};
\draw [-] (6) -- (7)node[midway, below] {$$};

\draw [white, line width=1.5mm] (A) -- (2) node[midway, left] {$$};
\draw [white, line width=1.5mm] (A) -- (3) node[midway, left] {$$};
\draw [white, line width=1.5mm] (A) -- (4) node[midway, left] {$$};
\draw [white, line width=1.5mm] (A) -- (5) node[midway, left] {$$};
\draw [white, line width=1.5mm] (A) -- (6) node[midway, left] {$$};
\draw [white, line width=1.5mm] (A) -- (8) node[midway, left] {$$};

\draw [-] (A) -- (2)node[midway, below] {$$};
\draw [-] (A) -- (3)node[midway, below] {$$};
\draw [-] (A) -- (4)node[midway, below] {$$};
\draw [-] (A) -- (5)node[midway, below] {$$};
\draw [-] (8) -- (6)node[midway, below] {$$};
\draw [white, line width=1.5mm] (A) -- (1) node[midway, left] {$$};
\draw [-] (A) -- (8)node[midway, below] {$$};
\draw [-] (A) -- (1)node[midway, below] {$$};
\draw [-] (8) -- (7)node[midway, below] {$$};
\draw [-] (5) -- (4)node[midway, below] {$$};
\draw [-] (B) -- (3)node[midway, below] {$$};
\draw [-] (A) -- (6)node[midway, below] {$$};
\draw [-] (A) -- (7)node[midway, below] {$$};

\node[anchor = north east]  at (1) {\small{$(x_{i-1},y_{j-1})$}};
\node[anchor = north west ]  at (2) {\small{$(x_{i},y_{j-1})$}};
\node[anchor = west ]  at (3) {\small{$(x_{i+1},y_{j-1})$}};
\node[anchor = west ]  at (4) {\small{$(x_{i+1},y_{j})$}};
\node[anchor = west ]  at (5) {\small{$(x_{i+1},y_{j+1})$}};
\node[anchor = south ]  at (6) {\small{$(x_{i},y_{j+1})$}};
\node[anchor = east ]  at (7) {\small{$(x_{i-1},y_{j+1})$}};
\node[anchor = north east ]  at (8) {\small{$(x_{i-1},y_{j})$}};

\node[anchor = south ]  at (A) {\small{$(A,y_j)$}};
\node[anchor = north ]  at (B) {\small{$(x_{i},B)$}};
.
\end{tikzpicture}
$$

It is then easy to see that $(x_{i-1},y_{j-1}), (x_{i+1}, y_{j-1}), (x_{i+1},y_{j+1}), (x_{i-1},y_{j+1})$ are all simple points which may be removed, yielding a digital space which is simple digitally equivalent to a digital $2$-sphere.  Thus all links are simple digitally equivalent to a digital $2$-sphere, and condition 1$^\prime$ is satisfied.

Finally we check condition 2$^\prime$ of Definition \ref{def: digital homology sphere}; that is, we show that $\widetilde{H}_i(S^1_8\ast S^1_8)=\Z$ for $i=3$ and $0$ otherwise. This is easily verified through computer software such as SAGE.  The details of the construction of $\widetilde{H}_i(S^1_8\ast S^1_8)=\Z$ in Sage are outlined in Appendix \ref{sec: construction}. With the naming convention used in the Appendix, we see that in SAGE we have

\begin{lstlisting}
cJoin = Join.clique_complex()

cJoin.homology()

{0: 0, 1: 0, 2: 0, 3: Z, 4: 0, 5: 0}.
\end{lstlisting}
\end{proof}

\section{Digital Hopf bundle}\label{sec: Digital Hopf bundle}

We saw in Theorem \ref{thm no tolerance group} that there are no non-simple digitally contractible tolerance space groups.  In the topological setting, the Hopf construction does not require a topological group, as such, but only a weaker kind of structure. In the tolerance setting, however, a circle does not even admit this weaker kind of structure.

\begin{definition} Let $X$ be a tolerance space.  We say that \emph{$X$ admits a multiplication with two-sided unit} if there is a map $m\colon X \times X \to X$ and a point $x_0\in X$ (the two-sided unit element) that satisfies $m(x,x_0)=x=m(x_0,x)$ for all $x\in X$.
\end{definition}

Although by Theorem \ref{thm no tolerance group}, we know that in general digital groups are simple digitally contractible, this fact is very easily seen in the case of a digital circle, so we include the following result here.

\begin{proposition}\label{prop: no circle multiplication} Let $S$ be a digital circle in the sense of Definition \ref{def: digital sphere}. Then $S$ does not admit a multiplication with two-sided unit.
\end{proposition}

\begin{proof} Suppose we have a continuous map $m\colon S\times S \to S$ and a point $s_0\in S$ that satisfies the definition.  Since $S$ satisfies Definition \ref{def: digital sphere}, the link of $s_0$ consists of two non-adjacent points $a$ and $b$. However, $(a,s_0)\sim_{S\times S} (x_0,b)$, hence the two-sided unit condition and continuity of $m$ give $a=m(a,s_0)\sim_{S}m(s_0,b)=b$, a contradiction.
\end{proof}

We work around this lack of a multiplication on the circle in the following way.  Our ``multiplication" takes in two points on $S^1_8$ and yields a point on $S^1_4$,  with $S^1_4$ given by

$$
\begin{tikzpicture}[scale=.6,
    decoration={markings,mark=at position 0.6 with {\arrow{triangle 60}}},
    ]

\node[inner sep=2pt, circle] (1) at (2,0) [draw] {};
\node[inner sep=2pt, circle] (i) at (0,2) [draw] {};
\node[inner sep=2pt, circle] (-1) at (-2,0) [draw] {};
\node[inner sep=2pt, circle] (-i) at (0,-2) [draw] {};

\draw [-] (1) -- (i) node[midway, left] {$$};
\draw [-] (1) -- (-i)node[midway, below] {$$};
\draw [-] (i) -- (-1)node[midway, below] {$$};
\draw [-] (-1) -- (-i)node[midway, below] {$$};

\node[anchor = west]  at (1) {\small{$1$}};
\node[anchor = south]  at (i) {\small{$i$}};
\node[anchor = east]  at (-1) {\small{$-1$}};
\node[anchor = north]  at (-i) {\small{$-i$}};.
\end{tikzpicture}
$$

We define $\mu\colon S^1_8\times S^1_8\to S^1_4$ as follows. View $S^1_8\times S^1_8$ as an identification space.  Then $\mu$ takes the same value on each diagonal or contour line according to the rules below:

$$
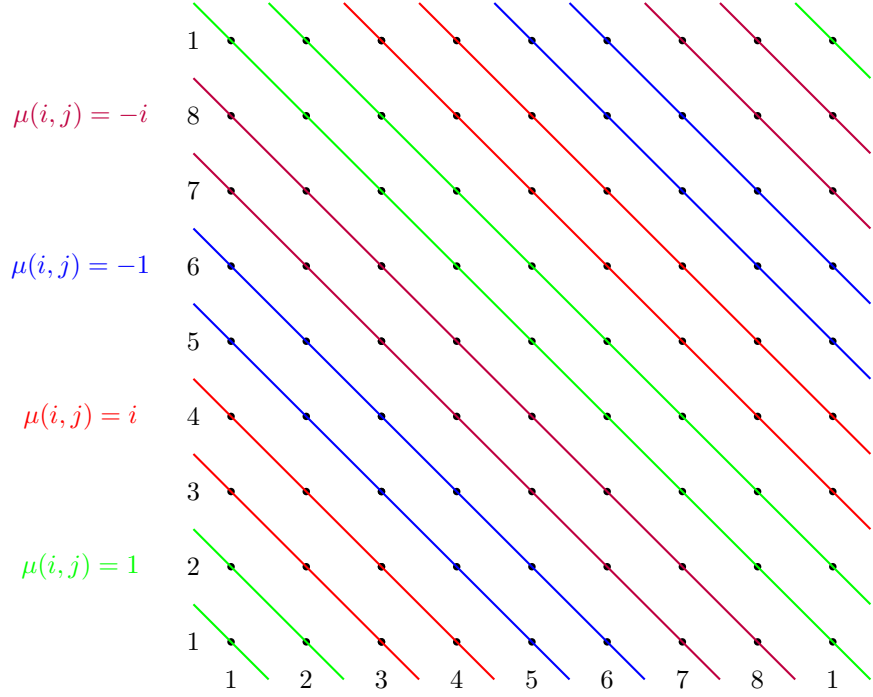
\begin{tikzpicture}[scale=1][thick]

    \node[inner sep=1pt, circle, fill=black] at (1,1) {};
    \node[inner sep=1pt, circle, fill=black] at (2,1) {};
    \node[inner sep=1pt, circle, fill=black] at (3,1) {};
    \node[inner sep=1pt, circle, fill=black] at (3,1) {};
    \node[inner sep=1pt, circle, fill=black] at (4,1) {};
    \node[inner sep=1pt, circle, fill=black] at (5,1) {};
    \node[inner sep=1pt, circle, fill=black] at (6,1) {};
    \node[inner sep=1pt, circle, fill=black] at (7,1) {};
    \node[inner sep=1pt, circle, fill=black] at (8,1) {};
    \node[inner sep=1pt, circle, fill=black] at (9,1) {};
    \node at (1,.5) {$1$};
    \node at (2,.5) {$2$};
    \node at (3,.5) {$3$};
    \node at (4,.5) {$4$};
    \node at (5,.5) {$5$};
    \node at (6,.5) {$6$};
    \node at (7,.5) {$7$};
    \node at (8,.5) {$8$};
    \node at (9,.5) {$1$};

    \node at (.5,1) {$1$};
    \node at (.5,2) {$2$};
    \node at (.5,3) {$3$};
    \node at (.5,4) {$4$};
    \node at (.5,5) {$5$};
    \node at (.5,6) {$6$};
    \node at (.5,7) {$7$};
    \node at (.5,8) {$8$};
    \node at (.5,9) {$1$};

    \node[inner sep=1pt, circle, fill=black] at (1,2) {};
    \node[inner sep=1pt, circle, fill=black] at (2,2) {};
    \node[inner sep=1pt, circle, fill=black] at (3,2) {};
    \node[inner sep=1pt, circle, fill=black] at (3,2) {};
    \node[inner sep=1pt, circle, fill=black] at (4,2) {};
    \node[inner sep=1pt, circle, fill=black] at (5,2) {};
    \node[inner sep=1pt, circle, fill=black] at (6,2) {};
    \node[inner sep=1pt, circle, fill=black] at (7,2) {};
    \node[inner sep=1pt, circle, fill=black] at (8,2) {};
    \node[inner sep=1pt, circle, fill=black] at (9,2) {};

    \node[inner sep=1pt, circle, fill=black] at (1,3) {};
    \node[inner sep=1pt, circle, fill=black] at (2,3) {};
    \node[inner sep=1pt, circle, fill=black] at (3,3) {};
    \node[inner sep=1pt, circle, fill=black] at (3,3) {};
    \node[inner sep=1pt, circle, fill=black] at (4,3) {};
    \node[inner sep=1pt, circle, fill=black] at (5,3) {};
    \node[inner sep=1pt, circle, fill=black] at (6,3) {};
    \node[inner sep=1pt, circle, fill=black] at (7,3) {};
    \node[inner sep=1pt, circle, fill=black] at (8,3) {};
    \node[inner sep=1pt, circle, fill=black] at (9,3) {};

    \node[inner sep=1pt, circle, fill=black] at (1,4) {};
    \node[inner sep=1pt, circle, fill=black] at (2,4) {};
    \node[inner sep=1pt, circle, fill=black] at (3,4) {};
    \node[inner sep=1pt, circle, fill=black] at (3,4) {};
    \node[inner sep=1pt, circle, fill=black] at (4,4) {};
    \node[inner sep=1pt, circle, fill=black] at (5,4) {};
    \node[inner sep=1pt, circle, fill=black] at (6,4) {};
    \node[inner sep=1pt, circle, fill=black] at (7,4) {};
    \node[inner sep=1pt, circle, fill=black] at (8,4) {};
    \node[inner sep=1pt, circle, fill=black] at (9,4) {};

    \node[inner sep=1pt, circle, fill=black] at (1,5) {};
    \node[inner sep=1pt, circle, fill=black] at (2,5) {};
    \node[inner sep=1pt, circle, fill=black] at (3,5) {};
    \node[inner sep=1pt, circle, fill=black] at (3,5) {};
    \node[inner sep=1pt, circle, fill=black] at (4,5) {};
    \node[inner sep=1pt, circle, fill=black] at (5,5) {};
    \node[inner sep=1pt, circle, fill=black] at (6,5) {};
    \node[inner sep=1pt, circle, fill=black] at (7,5) {};
    \node[inner sep=1pt, circle, fill=black] at (8,5) {};
    \node[inner sep=1pt, circle, fill=black] at (9,5) {};

    \node[inner sep=1pt, circle, fill=black] at (1,6) {};
    \node[inner sep=1pt, circle, fill=black] at (2,6) {};
    \node[inner sep=1pt, circle, fill=black] at (3,6) {};
    \node[inner sep=1pt, circle, fill=black] at (3,6) {};
    \node[inner sep=1pt, circle, fill=black] at (4,6) {};
    \node[inner sep=1pt, circle, fill=black] at (5,6) {};
    \node[inner sep=1pt, circle, fill=black] at (6,6) {};
    \node[inner sep=1pt, circle, fill=black] at (7,6) {};
    \node[inner sep=1pt, circle, fill=black] at (8,6) {};
    \node[inner sep=1pt, circle, fill=black] at (9,6) {};

    \node[inner sep=1pt, circle, fill=black] at (1,7) {};
    \node[inner sep=1pt, circle, fill=black] at (2,7) {};
    \node[inner sep=1pt, circle, fill=black] at (3,7) {};
    \node[inner sep=1pt, circle, fill=black] at (3,7) {};
    \node[inner sep=1pt, circle, fill=black] at (4,7) {};
    \node[inner sep=1pt, circle, fill=black] at (5,7) {};
    \node[inner sep=1pt, circle, fill=black] at (6,7) {};
    \node[inner sep=1pt, circle, fill=black] at (7,7) {};
    \node[inner sep=1pt, circle, fill=black] at (8,7) {};
    \node[inner sep=1pt, circle, fill=black] at (9,7) {};

    \node[inner sep=1pt, circle, fill=black] at (1,8) {};
    \node[inner sep=1pt, circle, fill=black] at (2,8) {};
    \node[inner sep=1pt, circle, fill=black] at (3,8) {};
    \node[inner sep=1pt, circle, fill=black] at (3,8) {};
    \node[inner sep=1pt, circle, fill=black] at (4,8) {};
    \node[inner sep=1pt, circle, fill=black] at (5,8) {};
    \node[inner sep=1pt, circle, fill=black] at (6,8) {};
    \node[inner sep=1pt, circle, fill=black] at (7,8) {};
    \node[inner sep=1pt, circle, fill=black] at (8,8) {};
    \node[inner sep=1pt, circle, fill=black] at (9,8) {};

    \node[inner sep=1pt, circle, fill=black] at (1,9) {};
    \node[inner sep=1pt, circle, fill=black] at (2,9) {};
    \node[inner sep=1pt, circle, fill=black] at (3,9) {};
    \node[inner sep=1pt, circle, fill=black] at (3,9) {};
    \node[inner sep=1pt, circle, fill=black] at (4,9) {};
    \node[inner sep=1pt, circle, fill=black] at (5,9) {};
    \node[inner sep=1pt, circle, fill=black] at (6,9) {};
    \node[inner sep=1pt, circle, fill=black] at (7,9) {};
    \node[inner sep=1pt, circle, fill=black] at (8,9) {};
    \node[inner sep=1pt, circle, fill=black] at (9,9) {};

    \draw[green, thick] (.5,1.5) -- (1.5,.5);
    \draw[green, thick] (.5,2.5) -- (2.5,.5);
    \draw[red, thick] (.5,3.5) -- (3.5,.5);
    \draw[red, thick] (.5,4.5) -- (4.5,.5);
    \draw[blue, thick] (.5,5.5) -- (5.5,.5);
    \draw[blue, thick] (.5,6.5) -- (6.5,.5);
    \draw[purple, thick] (.5,7.5) -- (7.5,.5);
    \draw[purple, thick] (.5,8.5) -- (8.5,.5);
    \draw[green, thick] (.5,9.5) -- (9.5,.5);
    \draw[green, thick] (1.5,9.5) -- (9.5,1.5);
    \draw[red, thick] (2.5,9.5) -- (9.5,2.5);
    \draw[red, thick] (3.5,9.5) -- (9.5,3.5);
    \draw[blue, thick] (4.5,9.5) -- (9.5,4.5);
    \draw[blue, thick] (5.5,9.5) -- (9.5,5.5);
    \draw[purple, thick] (6.5,9.5) -- (9.5,6.5);
    \draw[purple, thick] (7.5,9.5) -- (9.5,7.5);
    \draw[green, thick] (8.5,9.5) -- (9.5,8.5);

\node[purple] at (-1,8) {$\mu(i,j)=-i$};
\node[blue] at (-1,6) {$\mu(i,j)=-1$};
\node[red] at (-1,4) {$\mu(i,j)=i$};
\node[green] at (-1,2) {$\mu(i,j)=1$};

\end{tikzpicture}
$$
The reader can check that the fibers of the Hopf map wind around in the usual manner. Now the contour lines of $\mu$ come in pairs.  Hence points adjacent to a point in a given color are no more than $1$ adjacent color away.  For example, $\mu(3,4)=-1, \mu(6,2)=-i$, and $\mu(1,1)=1$. Thus $\mu$ is continuous.

\begin{definition}\label{def: digital mul} Let $\mu\colon S^1_8\times S^1_8\to S^1_4$ be defined as above.  The \emph{digital Hopf map} $H\mu\colon S^1_8\ast S^1_8\to S^2$ is given by
\begin{eqnarray*}
H\mu(x,B)&=&A\\
H\mu(A,y)&=&B\\
H\mu(x,y)&=&\mu(x,y).
\end{eqnarray*}
\end{definition}

To see that $H\mu$ is continuous, suppose that $(x,y)\sim (z,B)$ for $x\neq A$. If $y=B$, clearly $\mu(x,y)\sim\mu(s,B)$.   Otherwise, $\mu(x,y)\sim A=\mu(z,B)$ since $x\neq A$. A similar argument holds for $(x,y)\sim (A,z)$. Since $\mu$ was shown to be continuous above, $H\mu$ is continuous.

\begin{definition}\label{defn: digital fiber bundle} Let $E, B, F$ be tolerance spaces and $p\colon E \to B$ a surjective tolerance map.  We say that $(E,p,B,F)$ is a \emph{tolerance space fibre bundle} if
\begin{enumerate}
\item For every $b\in B, p^{-1}(b)\approx F$ are simple digitally equivalent.
\item For every $b\in B$, there exists a tolerance subspace $V_b\subseteq B$ with $b\in V_b$ such that $p^{-1}(V_b)\approx V_b\times F$.  In addition, there is a tolerance space map $\phi\colon p^{-1}(V_b)\to V_b \times F$ such that the following diagram commutes:
    $$
    \xymatrix{
    p^{-1}(V_b) \ar[rr]^{\phi}  \ar[rd]^{p} && V_b\times F \ar[ld]^{\pi_1}\\
    & V_b &
    }
    $$
\end{enumerate}
\end{definition}
Note that the map $\phi$ is not required to be an isomorphism.

\begin{theorem}\label{thm: digital Hopf fiber bundle} The map $H\mu\colon S^1_8\ast S^1_8\to S^1_4$ is a digital fiber bundle with fiber $F:=S^1_8$.
\end{theorem}

\begin{proof}
We verify conditions (1) and (2) of Definition \ref{defn: digital fiber bundle}.  Let $b\in S^2$.  If $b= A, B$, then it is easily seen that $(H\mu)^{-1}(A)\approx(H\mu)^{-1}(B)\approx S^1_8$.  Otherwise, suppose $b\neq A,B$.  Then it is also easily seen that $(H\mu)^{-1}(b)$ is given by

$$
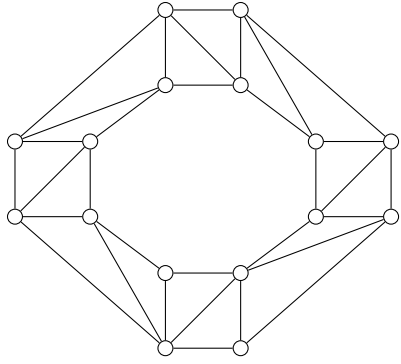
\begin{tikzpicture}[scale=.5,
    decoration={markings,mark=at position 0.6 with {\arrow{triangle 60}}},
    ]

\node[inner sep=2pt, circle] (1) at (4,3.5) [draw] {};
\node[inner sep=2pt, circle] (2) at (2,5) [draw] {};
\node[inner sep=2pt, circle] (3) at (0,5) [draw] {};
\node[inner sep=2pt, circle] (4) at (-2,3.5) [draw] {};
\node[inner sep=2pt, circle] (5) at (-2,1.5) [draw] {};
\node[inner sep=2pt, circle] (6) at (0,0) [draw] {};
\node[inner sep=2pt, circle] (7) at (2,0) [draw] {};
\node[inner sep=2pt, circle] (8) at (4,1.5) [draw] {};

\node[inner sep=2pt, circle] (1') at (6,3.5) [draw] {};
\node[inner sep=2pt, circle] (2') at (2,7) [draw] {};
\node[inner sep=2pt, circle] (3') at (0,7) [draw] {};
\node[inner sep=2pt, circle] (4') at (-4,3.5) [draw] {};
\node[inner sep=2pt, circle] (5') at (-4,1.5) [draw] {};
\node[inner sep=2pt, circle] (6') at (0,-2) [draw] {};
\node[inner sep=2pt, circle] (7') at (2,-2) [draw] {};
\node[inner sep=2pt, circle] (8') at (6,1.5) [draw] {};

\draw [-] (1) -- (2)node[midway, below] {$$};
\draw [-] (2) -- (3)node[midway, below] {$$};
\draw [-] (3) -- (4)node[midway, below] {$$};
\draw [-] (4) -- (5)node[midway, below] {$$};
\draw [-] (5) -- (6)node[midway, below] {$$};
\draw [-] (7) -- (6)node[midway, below] {$$};
\draw [-] (7) -- (8)node[midway, below] {$$};
\draw [-] (8) -- (1)node[midway, below] {$$};

\draw [-] (1') -- (2')node[midway, below] {$$};
\draw [-] (2') -- (3')node[midway, below] {$$};
\draw [-] (3') -- (4')node[midway, below] {$$};
\draw [-] (4') -- (5')node[midway, below] {$$};
\draw [-] (5') -- (6')node[midway, below] {$$};
\draw [-] (7') -- (6')node[midway, below] {$$};
\draw [-] (7') -- (8')node[midway, below] {$$};
\draw [-] (8') -- (1')node[midway, below] {$$};

\draw [-] (1) -- (1')node[midway, below] {$$};
\draw [-] (2) -- (2')node[midway, below] {$$};
\draw [-] (3) -- (3')node[midway, below] {$$};
\draw [-] (4) -- (4')node[midway, below] {$$};
\draw [-] (5) -- (5')node[midway, below] {$$};
\draw [-] (6) -- (6')node[midway, below] {$$};
\draw [-] (7) -- (7')node[midway, below] {$$};
\draw [-] (8) -- (8')node[midway, below] {$$};

\draw [-] (1) -- (2')node[midway, below] {$$};
\draw [-] (2) -- (3')node[midway, below] {$$};
\draw [-] (3) -- (4')node[midway, below] {$$};
\draw [-] (4) -- (5')node[midway, below] {$$};
\draw [-] (5) -- (6')node[midway, below] {$$};
\draw [-] (7) -- (6')node[midway, below] {$$};
\draw [-] (7) -- (8')node[midway, below] {$$};
\draw [-] (8) -- (1')node[midway, below] {$$};

\end{tikzpicture}
$$
which is clearly simple digitally equivalent to $S^1_8$.  To verify condition (2), let $U_1:=C_AS^1_4:=S^1_4*\{A\}$ and $U_2:=C_BS^1_4:=S^1_4*\{B\}$, the northern and southern hemisphere of $S^2$, respectively.  Then $H\mu^{-1}(U_1)=(C_BS^1_8)\times S^1_8$.  We thus need to show that $(C_BS^1_8)\times S^1_8\approx (C_AS^1_4)\times S^1_8$. This is shown explicitly in Lemma \ref{lem: simple equivalence of product} below. Thus, supposing $(C_BS^1_8)\times S^1_8\approx (C_AS^1_4)\times S^1_8$, we have
$$
H\mu^{-1}(U_1)=(C_BS^1_8)\times S^1_8\approx (C_AS^1_4)\times S^1_8=U_1\times F.
$$

To find the map $\phi\colon (H\mu)^{-1}(U_1)\to U_1\times F$, observe that there is a tolerance space map $\overline{\phi}\colon S^1_8\to S^1_4$ given by $\overline{\phi}(1)=\overline{\phi}(2)=1, \overline{\phi}(3)=\overline{\phi}(4)=i, \overline{\phi}(5)=\overline{\phi}(6)=-1,$ and $\overline{\phi}(7)=\overline{\phi}(8)=-i$.  This may be extended to a map $\phi\colon (C_BS^1_8)\times S^1_8\to (C_AS^1_4)\times S^1_8$ making the following diagram commute:
    $$
    \xymatrix{
    (C_BS^1_8)\times S^1_8 \ar[rr]^{\phi}  \ar[rd]^{H\mu} && (C_AS^1_4)\times S^1_8 \ar[ld]^{\pi_1}\\
    & (C_AS^1_4) &
    }
    $$
The exact same argument holds for $U_2$. Thus $H\mu$ is a digital fiber bundle.
\end{proof}

\begin{lemma}\label{lem: simple equivalence of product} There is a simple digital homotopy equivalence $(C_BS^1_8)\times S^1_8\approx (C_AS^1_4)\times S^1_8$.
\end{lemma}

\begin{proof} We explicitly find a sequence of simple points to remove from $(C_BS^1_8)\times S^1_8$ to obtain $S^1_8$.  Let $v_1, v_2, \ldots, v_8$ denote the points of $S^1_8$ in $C_BS^1_8$ and $u_1, u_2, \ldots, u_8$ denote the points of $S^1_8$ in the second factor.  We claim that the removal of the points
\begin{eqnarray*}
(v_1,u_1), (v_1, u_2), &\ldots &, (v_1,u_8)\\
(v_2,u_1), (v_2, u_2), &\ldots &, (v_2,u_8)\\
\vdots &\ddots& \vdots \\
(v_8,u_1), (v_8, u_2), &\ldots &, (v_8,u_8)\\
\end{eqnarray*}
in this order constitutes a removal of simple points, where each point is simple in the space with all previous points removed. When we claim below that a point is simple, we mean that it is simple in the space with all previous points removed. We first observe that the link of $(v_1, u_1)$ is given by

$$
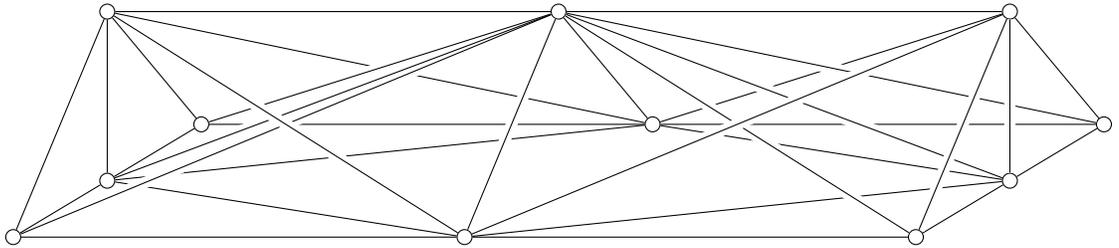
\begin{tikzpicture}[scale=.5,
    decoration={markings,mark=at position 0.6 with {\arrow{triangle 60}}},
    ]

\node[inner sep=2pt, circle] (88) at (0,0) [draw] {};
\node[inner sep=2pt, circle] (81) at (12,0) [draw] {};
\node[inner sep=2pt, circle] (82) at (24,0) [draw] {};
\node[inner sep=2pt, circle] (18) at (2.5,1.5) [draw] {};
\node[inner sep=2pt, circle] (12) at (26.5,1.5) [draw] {};
\node[inner sep=2pt, circle] (28) at (5,3) [draw] {};
\node[inner sep=2pt, circle] (21) at (17,3) [draw] {};
\node[inner sep=2pt, circle] (22) at (29,3) [draw] {};
\node[inner sep=2pt, circle] (B8) at (2.5,6) [draw] {};
\node[inner sep=2pt, circle] (B1) at (14.5,6) [draw] {};
\node[inner sep=2pt, circle] (B2) at (26.5,6) [draw] {};

\draw [-] (21) -- (B2)node[midway, below] {$$};
\draw [-] (21) -- (12)node[midway, below] {$$};

\draw [-] (18) -- (21)node[midway, below] {$$};

\draw [-] (28) -- (21)node[midway, below] {$$};

\draw [-] (81) -- (12)node[midway, below] {$$};
\draw [-] (81) -- (18)node[midway, below] {$$};
\draw [-] (18) -- (B8)node[midway, below] {$$};
\draw [-] (B8) -- (21)node[midway, below] {$$};

\draw [-] (82) -- (12)node[midway, below] {$$};
\draw [-] (88) -- (B8)node[midway, below] {$$};
\draw [-] (22) -- (12)node[midway, below] {$$};
\draw [-] (B2) -- (22)node[midway, below] {$$};

\draw [white, line width=1.5mm] (22) -- (B1) node[midway, left] {$$};
\draw [-] (21) -- (22)node[midway, below] {$$};
\draw [white, line width=1.5mm] (23.833333,2.5) -- (B1) node[midway, left] {$$};
\draw [white, line width=1.5mm] (8.1666,4) -- (B1) node[midway, left] {$$};
\draw [white, line width=1.5mm] (5.16666,2.5) -- (B1) node[midway, left] {$$};

\draw [white, line width=1.5mm] (B1) -- (82) node[midway, left] {$$};
\draw [white, line width=1.5mm] (88) -- (B1) node[midway, left] {$$};
\draw [white, line width=1.5mm] (B1) -- (81) node[midway, left] {$$};

\draw [-] (22) -- (B1)node[midway, below] {$$};
\draw [white, line width=1.5mm] (B2) -- (12) node[midway, left] {$$};
\draw [-] (12) -- (B1)node[midway, below] {$$};
\draw [-] (28) -- (B1)node[midway, below] {$$};
\draw [-] (18) -- (B1)node[midway, below] {$$};

\draw [-] (B1) -- (82)node[midway, below] {$$};
\draw [-] (88) -- (B1)node[midway, below] {$$};
\draw [white, line width=1.5mm] (4.0833,5) -- (10.4167,1) node[midway, left] {$$};
\draw [-] (B1) -- (81)node[midway, below] {$$};
\draw [-] (B8) -- (81)node[midway, below] {$$};
\draw [white, line width=1.5mm] (24.083,5) -- (14.417,1) node[midway, left] {$$};
\draw [-] (B2) -- (81)node[midway, below] {$$};

\draw [white, line width=1.5mm] (B2) -- (82) node[midway, left] {$$};
\draw [-] (B2) -- (12)node[midway, below] {$$};
\draw [-] (B2) -- (82)node[midway, below] {$$};
\draw [-] (18) -- (28)node[midway, below] {$$};
\draw [-] (28) -- (B8)node[midway, below] {$$};
\draw [-] (B8) -- (B1)node[midway, below] {$$};
\draw [-] (B1) -- (B2)node[midway, below] {$$};
\draw [-] (88) -- (81)node[midway, below] {$$};
\draw [-] (88) -- (18)node[midway, below] {$$};
\draw [-] (82) -- (81)node[midway, below] {$$};

\draw [-] (B1) -- (21)node[midway, below] {$$};

\end{tikzpicture}
$$
which is seen to be simple digitally contractible. After removing $(v_1,u_1)$, the link of $(v_1,u_2)$ is
$$
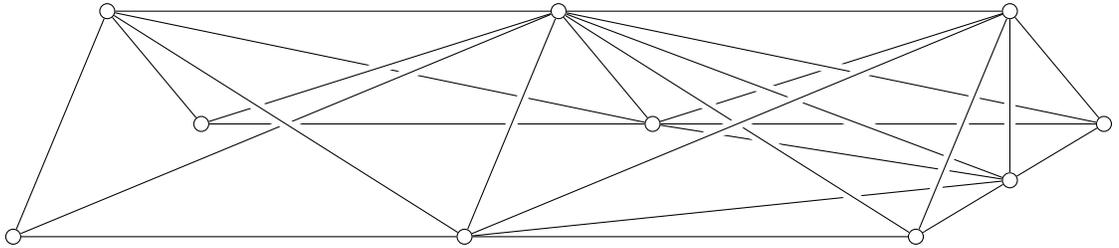
\begin{tikzpicture}[scale=.5,
    decoration={markings,mark=at position 0.6 with {\arrow{triangle 60}}},
    ]

\node[inner sep=2pt, circle] (88) at (0,0) [draw] {};
\node[inner sep=2pt, circle] (81) at (12,0) [draw] {};
\node[inner sep=2pt, circle] (82) at (24,0) [draw] {};
\node[inner sep=2pt, circle] (12) at (26.5,1.5) [draw] {};
\node[inner sep=2pt, circle] (28) at (5,3) [draw] {};
\node[inner sep=2pt, circle] (21) at (17,3) [draw] {};
\node[inner sep=2pt, circle] (22) at (29,3) [draw] {};
\node[inner sep=2pt, circle] (B8) at (2.5,6) [draw] {};
\node[inner sep=2pt, circle] (B1) at (14.5,6) [draw] {};
\node[inner sep=2pt, circle] (B2) at (26.5,6) [draw] {};

\draw [-] (21) -- (B2)node[midway, below] {$$};
\draw [-] (21) -- (12)node[midway, below] {$$};

\draw [-] (28) -- (21)node[midway, below] {$$};

\draw [-] (81) -- (12)node[midway, below] {$$};

\draw [-] (B8) -- (21)node[midway, below] {$$};

\draw [-] (82) -- (12)node[midway, below] {$$};
\draw [-] (88) -- (B8)node[midway, below] {$$};
\draw [-] (22) -- (12)node[midway, below] {$$};
\draw [-] (B2) -- (22)node[midway, below] {$$};

\draw [white, line width=1.5mm] (22) -- (B1) node[midway, left] {$$};
\draw [-] (21) -- (22)node[midway, below] {$$};
\draw [white, line width=1.5mm] (23.833333,2.5) -- (B1) node[midway, left] {$$};
\draw [white, line width=1.5mm] (8.1666,4) -- (B1) node[midway, left] {$$};

\draw [white, line width=1.5mm] (B1) -- (82) node[midway, left] {$$};
\draw [white, line width=1.5mm] (88) -- (B1) node[midway, left] {$$};
\draw [white, line width=1.5mm] (B1) -- (81) node[midway, left] {$$};

\draw [-] (22) -- (B1)node[midway, below] {$$};
\draw [white, line width=1.5mm] (B2) -- (12) node[midway, left] {$$};
\draw [-] (12) -- (B1)node[midway, below] {$$};
\draw [-] (28) -- (B1)node[midway, below] {$$};

\draw [-] (B1) -- (82)node[midway, below] {$$};
\draw [-] (88) -- (B1)node[midway, below] {$$};
\draw [white, line width=1.5mm] (4.0833,5) -- (10.4167,1) node[midway, left] {$$};
\draw [-] (B1) -- (81)node[midway, below] {$$};
\draw [-] (B8) -- (81)node[midway, below] {$$};
\draw [white, line width=1.5mm] (24.083,5) -- (14.417,1) node[midway, left] {$$};
\draw [-] (B2) -- (81)node[midway, below] {$$};

\draw [white, line width=1.5mm] (B2) -- (82) node[midway, left] {$$};
\draw [-] (B2) -- (12)node[midway, below] {$$};
\draw [-] (B2) -- (82)node[midway, below] {$$};
\draw [-] (28) -- (B8)node[midway, below] {$$};
\draw [-] (B8) -- (B1)node[midway, below] {$$};
\draw [-] (B1) -- (B2)node[midway, below] {$$};
\draw [-] (88) -- (81)node[midway, below] {$$};
\draw [-] (82) -- (81)node[midway, below] {$$};

\draw [-] (B1) -- (21)node[midway, below] {$$};

\end{tikzpicture}
$$
which is also seen to be simple digitally contractible. Hence remove the simple point $(v_1,u_2)$, and observe that the link of $(v_1, u_3), \ldots, (v_1,u_7)$ is isomoprphic to the link of $(v_1,u_2)$. After removal of these simple points, the link of $(v_1,u_8)$ is
$$
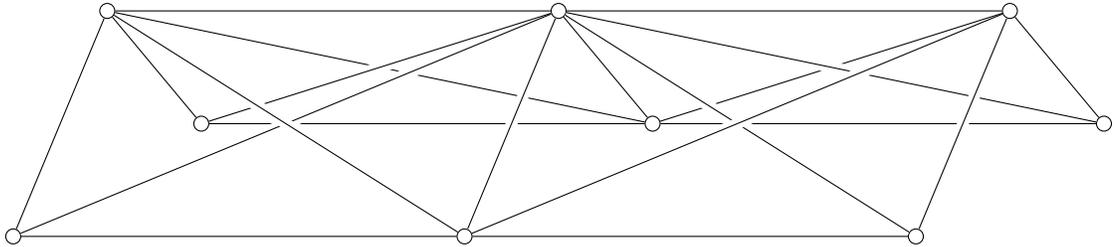
\begin{tikzpicture}[scale=.5,
    decoration={markings,mark=at position 0.6 with {\arrow{triangle 60}}},
    ]

\node[inner sep=2pt, circle] (88) at (0,0) [draw] {};
\node[inner sep=2pt, circle] (81) at (12,0) [draw] {};
\node[inner sep=2pt, circle] (82) at (24,0) [draw] {};
\node[inner sep=2pt, circle] (28) at (5,3) [draw] {};
\node[inner sep=2pt, circle] (21) at (17,3) [draw] {};
\node[inner sep=2pt, circle] (22) at (29,3) [draw] {};
\node[inner sep=2pt, circle] (B8) at (2.5,6) [draw] {};
\node[inner sep=2pt, circle] (B1) at (14.5,6) [draw] {};
\node[inner sep=2pt, circle] (B2) at (26.5,6) [draw] {};

\draw [-] (21) -- (B2)node[midway, below] {$$};

\draw [-] (28) -- (21)node[midway, below] {$$};

\draw [-] (B8) -- (21)node[midway, below] {$$};

\draw [-] (88) -- (B8)node[midway, below] {$$};
\draw [-] (B2) -- (22)node[midway, below] {$$};

\draw [white, line width=1.5mm] (22) -- (B1) node[midway, left] {$$};
\draw [-] (21) -- (22)node[midway, below] {$$};
\draw [white, line width=1.5mm] (8.1666,4) -- (B1) node[midway, left] {$$};

\draw [white, line width=1.5mm] (B1) -- (82) node[midway, left] {$$};
\draw [white, line width=1.5mm] (88) -- (B1) node[midway, left] {$$};
\draw [white, line width=1.5mm] (B1) -- (81) node[midway, left] {$$};

\draw [-] (22) -- (B1)node[midway, below] {$$};
\draw [-] (28) -- (B1)node[midway, below] {$$};

\draw [-] (B1) -- (82)node[midway, below] {$$};
\draw [-] (88) -- (B1)node[midway, below] {$$};
\draw [white, line width=1.5mm] (4.0833,5) -- (10.4167,1) node[midway, left] {$$};
\draw [-] (B1) -- (81)node[midway, below] {$$};
\draw [-] (B8) -- (81)node[midway, below] {$$};
\draw [white, line width=1.5mm] (24.083,5) -- (14.417,1) node[midway, left] {$$};
\draw [-] (B2) -- (81)node[midway, below] {$$};

\draw [white, line width=1.5mm] (B2) -- (82) node[midway, left] {$$};
\draw [-] (B2) -- (82)node[midway, below] {$$};
\draw [-] (28) -- (B8)node[midway, below] {$$};
\draw [-] (B8) -- (B1)node[midway, below] {$$};
\draw [-] (B1) -- (B2)node[midway, below] {$$};
\draw [-] (88) -- (81)node[midway, below] {$$};
\draw [-] (82) -- (81)node[midway, below] {$$};

\draw [-] (B1) -- (21)node[midway, below] {$$};

\end{tikzpicture}
$$
which is simple digitally contractible so we remove it. Now the link of $(v_2,u_1)$ is seen to be

$$
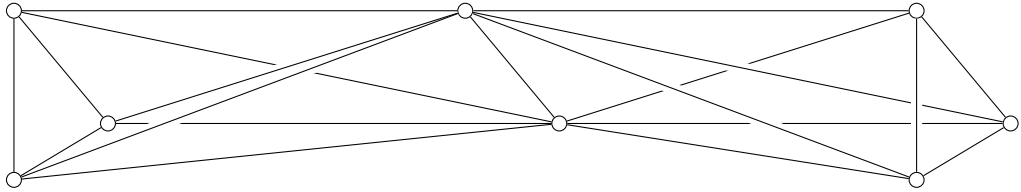
\begin{tikzpicture}[scale=.5,
    decoration={markings,mark=at position 0.6 with {\arrow{triangle 60}}},
    ]

\node[inner sep=2pt, circle] (18) at (2.5,1.5) [draw] {};
\node[inner sep=2pt, circle] (12) at (26.5,1.5) [draw] {};
\node[inner sep=2pt, circle] (28) at (5,3) [draw] {};
\node[inner sep=2pt, circle] (21) at (17,3) [draw] {};
\node[inner sep=2pt, circle] (22) at (29,3) [draw] {};
\node[inner sep=2pt, circle] (B8) at (2.5,6) [draw] {};
\node[inner sep=2pt, circle] (B1) at (14.5,6) [draw] {};
\node[inner sep=2pt, circle] (B2) at (26.5,6) [draw] {};

\draw [-] (21) -- (B2)node[midway, below] {$$};
\draw [-] (21) -- (12)node[midway, below] {$$};

\draw [-] (18) -- (21)node[midway, below] {$$};

\draw [-] (28) -- (21)node[midway, below] {$$};

\draw [-] (18) -- (B8)node[midway, below] {$$};
\draw [-] (B8) -- (21)node[midway, below] {$$};

\draw [-] (22) -- (12)node[midway, below] {$$};
\draw [-] (B2) -- (22)node[midway, below] {$$};

\draw [white, line width=1.5mm] (22) -- (B1) node[midway, left] {$$};
\draw [-] (21) -- (22)node[midway, below] {$$};
\draw [white, line width=1.5mm] (23.833333,2.5) -- (B1) node[midway, left] {$$};
\draw [white, line width=1.5mm] (8.1666,4) -- (B1) node[midway, left] {$$};
\draw [white, line width=1.5mm] (5.16666,2.5) -- (B1) node[midway, left] {$$};

\draw [-] (22) -- (B1)node[midway, below] {$$};
\draw [white, line width=1.5mm] (B2) -- (12) node[midway, left] {$$};
\draw [-] (12) -- (B1)node[midway, below] {$$};
\draw [-] (28) -- (B1)node[midway, below] {$$};
\draw [-] (18) -- (B1)node[midway, below] {$$};

\draw [-] (B2) -- (12)node[midway, below] {$$};
\draw [-] (18) -- (28)node[midway, below] {$$};
\draw [-] (28) -- (B8)node[midway, below] {$$};
\draw [-] (B8) -- (B1)node[midway, below] {$$};
\draw [-] (B1) -- (B2)node[midway, below] {$$};

\draw [-] (B1) -- (21)node[midway, below] {$$};

\end{tikzpicture}
$$
This is simple digitally contractible, and we remove it.  The link of $(v_2,u_2), \ldots (v_2,u_7)$ is given by
$$
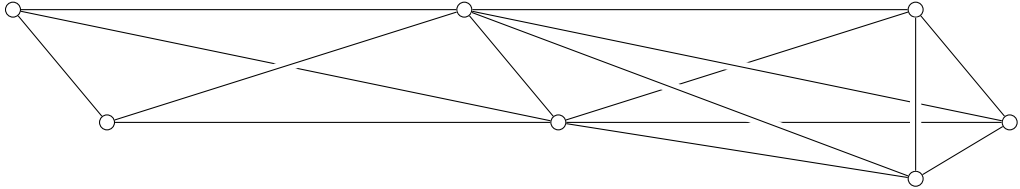
\begin{tikzpicture}[scale=.5,
    decoration={markings,mark=at position 0.6 with {\arrow{triangle 60}}},
    ]

\node[inner sep=2pt, circle] (12) at (26.5,1.5) [draw] {};
\node[inner sep=2pt, circle] (28) at (5,3) [draw] {};
\node[inner sep=2pt, circle] (21) at (17,3) [draw] {};
\node[inner sep=2pt, circle] (22) at (29,3) [draw] {};
\node[inner sep=2pt, circle] (B8) at (2.5,6) [draw] {};
\node[inner sep=2pt, circle] (B1) at (14.5,6) [draw] {};
\node[inner sep=2pt, circle] (B2) at (26.5,6) [draw] {};

\draw [-] (21) -- (B2)node[midway, below] {$$};
\draw [-] (21) -- (12)node[midway, below] {$$};

\draw [-] (28) -- (21)node[midway, below] {$$};

\draw [-] (B8) -- (21)node[midway, below] {$$};

\draw [-] (22) -- (12)node[midway, below] {$$};
\draw [-] (B2) -- (22)node[midway, below] {$$};

\draw [white, line width=1.5mm] (22) -- (B1) node[midway, left] {$$};
\draw [-] (21) -- (22)node[midway, below] {$$};
\draw [white, line width=1.5mm] (23.833333,2.5) -- (B1) node[midway, left] {$$};
\draw [white, line width=1.5mm] (8.1666,4) -- (B1) node[midway, left] {$$};

\draw [-] (22) -- (B1)node[midway, below] {$$};
\draw [white, line width=1.5mm] (B2) -- (12) node[midway, left] {$$};
\draw [-] (12) -- (B1)node[midway, below] {$$};
\draw [-] (28) -- (B1)node[midway, below] {$$};

\draw [-] (B2) -- (12)node[midway, below] {$$};
\draw [-] (28) -- (B8)node[midway, below] {$$};
\draw [-] (B8) -- (B1)node[midway, below] {$$};
\draw [-] (B1) -- (B2)node[midway, below] {$$};

\draw [-] (B1) -- (21)node[midway, below] {$$};

\end{tikzpicture}
$$
and after removing these points, the link of $(v_2,u_7)$ is
$$
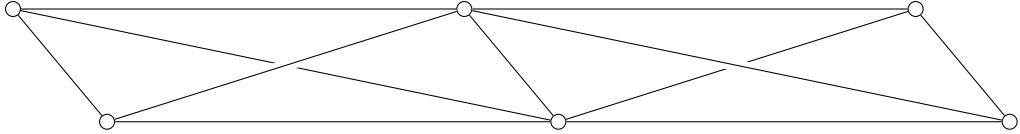
\begin{tikzpicture}[scale=.5,
    decoration={markings,mark=at position 0.6 with {\arrow{triangle 60}}},
    ]

\node[inner sep=2pt, circle] (28) at (5,3) [draw] {};
\node[inner sep=2pt, circle] (21) at (17,3) [draw] {};
\node[inner sep=2pt, circle] (22) at (29,3) [draw] {};
\node[inner sep=2pt, circle] (B8) at (2.5,6) [draw] {};
\node[inner sep=2pt, circle] (B1) at (14.5,6) [draw] {};
\node[inner sep=2pt, circle] (B2) at (26.5,6) [draw] {};

\draw [-] (21) -- (B2)node[midway, below] {$$};

\draw [-] (28) -- (21)node[midway, below] {$$};

\draw [-] (B8) -- (21)node[midway, below] {$$};

\draw [-] (B2) -- (22)node[midway, below] {$$};

\draw [white, line width=1.5mm] (22) -- (B1) node[midway, left] {$$};
\draw [-] (21) -- (22)node[midway, below] {$$};
\draw [white, line width=1.5mm] (8.1666,4) -- (B1) node[midway, left] {$$};

\draw [-] (22) -- (B1)node[midway, below] {$$};
\draw [-] (28) -- (B1)node[midway, below] {$$};

\draw [-] (28) -- (B8)node[midway, below] {$$};
\draw [-] (B8) -- (B1)node[midway, below] {$$};
\draw [-] (B1) -- (B2)node[midway, below] {$$};

\draw [-] (B1) -- (21)node[midway, below] {$$};

\end{tikzpicture}
$$
Now observe that the links of $(v_3,u_i),(v_4, u_i), \ldots, (v_7,u_i)$ follow the same pattern as the three links of $v_2$ above. After removing all these points we are left with $S^1_8\times K_2$ which is clearly simple homotopy equivalent to $S^1_8$.  The same argument shows that $(C_AS^1_4)\times S^1_8\approx S^1_8$.  Thus $(C_BS^1_8)\times S^1_8\approx (C_AS^1_4)\times S^1_8$.
\end{proof}

\appendix
\section{Construction of $S_8^1\ast S_8^1$}\label{sec: construction}
We share the code along with comments for constructing $S_8^1\ast S_8^1$ in Sage. Write $X = S^1_8=\{x1, \ldots, x8\}$ and  $Y = S^1_8=\{y1, \ldots, y8\}$
with apex in $A\in CX$ and apex $B\in CY$. Input these values in SAGE as numeric values

\begin{lstlisting}
A = 0
x1 = 1
x2 = 2
x3 = 3
x4 = 4
x5 = 5
x6 = 6
x7 = 7
x8 = 8
y1 = 9
y2 = 10
y3 = 11
y4 = 12
y5 = 13
y6 = 14
y7 = 15
y8 = 16
B = 17
\end{lstlisting}

Next we give the adjacency relations among the values in $X$ and $Y$ by specifying the neighbors of each point.

\begin{lstlisting}
Nx1 = [x8, x2]
Nx2 = [x1, x3]
Nx3 = [x2, x4]
Nx4 = [x3, x5]
Nx5 = [x4, x6]
Nx6 = [x5, x7]
Nx7 = [x6, x8]
Nx8 = [x7, x1]
Ny1 = [y8, y2]
Ny2 = [y1, y3]
Ny3 = [y2, y4]
Ny4 = [y3, y5]
Ny5 = [y4, y6]
Ny6 = [y5, y7]
Ny7 = [y6, y8]
Ny8 = [y7, y1]
\end{lstlisting}

We now define the digital images $X$ and $Y$ into SAGE.

\begin{lstlisting}
X = Graph({x1:Nx1, x2:Nx2, x3:Nx3, x4:Nx4, x5:Nx5, x6:Nx6, x7:Nx7,
x8:Nx8})
Y = Graph({y1:Ny1, y2:Ny2, y3:Ny3, y4:Ny4, y5:Ny5, y6:Ny6, y7:Ny7,
y8:Ny8})
\end{lstlisting}

From here we generate $CX$ and $CY$ with, of course, the proper adjacency relations.

\begin{lstlisting}
CX = Graph({x1:Nx1, x2:Nx2, x3:Nx3, x4:Nx4, x5:Nx5, x6:Nx6, x7:Nx7,
x8:Nx8})
for i in range(x1,x8+1):
    CX.add_edge(A, i)
CY = Graph({y1:Ny1, y2:Ny2, y3:Ny3, y4:Ny4, y5:Ny5, y6:Ny6, y7:Ny7,
y8:Ny8})
for i in range(y1,y8+1):
    CY.add_edge(B, i)
\end{lstlisting}

Next we construct the products $CX\times Y$ and $X\times CY$.

\begin{lstlisting}
Bottom = CX.strong_product(Y)
Top = X.strong_product(CY)
\end{lstlisting}

Finally, the join $S^1_8\ast S^1_8$ is the union of $CX \times Y$ and $X x\times CY$.

\begin{lstlisting}
Join = Bottom.union(Top)
\end{lstlisting}

\bibliographystyle{amsplain}
\bibliography{Hopf}

\providecommand{\bysame}{\leavevmode\hbox to3em{\hrulefill}\thinspace}
\providecommand{\MR}{\relax\ifhmode\unskip\space\fi MR }
\providecommand{\MRhref}[2]{%
  \href{http://www.ams.org/mathscinet-getitem?mr=#1}{#2}
}
\providecommand{\href}[2]{#2}
\begin{thebibliography}{10}

\bibitem{Boxer2005}
L.~Boxer, \emph{Properties of digital homotopy}, J. Math. Imaging Vision
  \textbf{22} (2005), no.~1, 19--26.

\bibitem{BoxerKaraca2008}
L.~Boxer and I.~Karaca, \emph{The classification of digital covering spaces},
  J. Math. Imaging Vision \textbf{32} (2008), no.~1, 23--29.

\bibitem{BoxerKaraca2010}
\bysame, \emph{Some properties of digital covering spaces}, J. Math. Imaging
  Vision \textbf{37} (2010), no.~1, 17--26.

\bibitem{EgeKaraca2017}
O.~Ege and I.~Karaca, \emph{Digital fibrations}, Proc. Nat. Acad. Sci. India
  Sect. A \textbf{87} (2017), no.~1, 109--114.

\bibitem{Evako2015}
A.~V. Evako, \emph{Classification of digital {$n$}-manifolds}, Discrete Appl.
  Math. \textbf{181} (2015), 289--296.

\bibitem{Ivan1991}
A.~V. Ivashchenko, \emph{Contractible transformations do not chage the homology
  groups of graphs}, Discrete Mathematics \textbf{126} (1994), 159--170.

\bibitem{Evako94}
\bysame, \emph{Some properties of contractible transformations on graphs},
  Discrete Math. \textbf{133} (1994), no.~1-3, 139--145.

\bibitem{LOS19a}
G.~Lupton, J.~Oprea, and N.~A. Scoville, \emph{Homotopy theory in digital
  topology}, arXiv:1905.07783 [math.AT], 2019.

\bibitem{Matveev2007}
S.~Matveev, \emph{Algorithmic topology and classification of 3-manifolds},
  second ed., Algorithms and Computation in Mathematics, vol.~9, Springer,
  Berlin, 2007.

\bibitem{MuirWarner1980}
A.~Muir and M.~W. Warner, \emph{Homogeneous tolerance spaces}, Czechoslovak
  Math. J. \textbf{30(105)} (1980), no.~1, 118--126.

\bibitem{Peters12}
J.~F. Peters and P.~Wasilewski, \emph{Tolerance spaces: origins, theoretical
  aspects and applications}, Inform. Sci. \textbf{195} (2012), 211--225.

\bibitem{Poston}
T.~Poston, \emph{Fuzzy geometry}, 1972, Thesis (Ph.D.)--University of Warwick.

\bibitem{Sossinsky1986}
A.~B. Sossinsky, \emph{Tolerance space theory and some applications}, Acta
  Appl. Math. \textbf{5} (1986), no.~2, 137--167.

\bibitem{Whitehead1950}
J.~H.~C. Whitehead, \emph{Simple homotopy types}, Amer. J. Math. \textbf{72}
  (1950), 1--57.

\bibitem{Zeeman}
E.C. Zeeman, \emph{The topology of the brain and visual perception}, Topology
  of 3-manifolds and related topics (Proc. The Univ. of Georgia Institute,
  1961) \textbf{Prentice-Hall} (1962), 240--256.

\end{thebibliography}

\end{document}